\documentclass{amsart}
\usepackage{amssymb,amsmath,amscd,enumerate,verbatim}
\usepackage[colorlinks=true,linkcolor=black,citecolor=black]{hyperref}
\usepackage[all]{xy}
\usepackage{cleveref}
\usepackage{graphicx}
\usepackage{algorithm}
\usepackage{algpseudocode}

\numberwithin{equation}{section}

\newtheorem{thm}{\bf Theorem}[section]
\newtheorem{lem}[thm]{\bf Lemma}
\newtheorem{conj}[thm]{\bf Conjecture}
\newtheorem{cor}[thm]{\bf Corollary}
\newtheorem{prop}[thm]{\bf Proposition}

\theoremstyle{definition}
\newtheorem{defn}[thm]{Definition}
\newtheorem{constr}[thm]{Construction}
\newtheorem{ex}[thm]{Example}
\newtheorem{notn}[thm]{Notation}
\newtheorem{rem}[thm]{Remark}

\newtheorem*{thm*}{Theorem}

\DeclareMathOperator{\Tor}{Tor}
\DeclareMathOperator{\reg}{reg}

\DeclareMathOperator{\Inc}{Inc}
\DeclareMathOperator{\ind}{ind}
\DeclareMathOperator{\indmatch}{indmatch}

\DeclareMathOperator{\Sym}{Sym}

\DeclareMathOperator{\Msupp}{Msupp}

\newcommand{\Z}{{\mathbb Z}}
\newcommand{\N}{{\mathbb N}}

\newcommand{\R}{{\mathbb R}}

\def\Icc{{\mathcal I}}
\def\Jcc{{\mathcal J}}

\newcommand{\kk}{\Bbbk}

\begin{document}

\title[Asymptotic regularity of invariant chains]{Asymptotic regularity of invariant chains \\ of edge ideals}

\author[D.T. Hoang]{Do Trong Hoang}
\address{School of Applied Mathematics and Informatics, Hanoi University of Science and Technology, 1 Dai Co Viet, Hai Ba Trung, Hanoi, Vietnam}
\email{hoang.dotrong@hust.edu.vn}
\author[H.D. Nguyen]{Hop D. Nguyen}
\address{Institute of Mathematics, Vietnam Academy of Science and Technology, 18 Hoang Quoc Viet, 10307 Hanoi, Vietnam}
\email{ngdhop@gmail.com}
\author[Q.H. Tran]{Quang Hoa Tran}
\address{University of Education, Hue University, 34 Le Loi St., Hue City, Viet Nam}
\email{tranquanghoa@hueuni.edu.vn}

\begin{abstract}
We study chains of nonzero edge ideals that are invariant under the action of the monoid $\Inc$ of increasing functions on the positive integers. We prove that the sequence of Castelnuovo--Mumford regularity of ideals in such a chain is eventually constant with limit either 2 or 3, and we determine explicitly when the constancy behaviour sets in. This provides further evidence to a conjecture on the asymptotic linearity of the regularity of $\Inc$-invariant chains of homogeneous ideals. The proofs reveal unexpected combinatorial properties of $\Inc$-invariant chains of edge ideals.
\end{abstract}

\makeatletter
\@namedef{subjclassname@2020}{%
	\textup{2020} Mathematics Subject Classification}
\makeatother


\maketitle

\section{Introduction}

Let $\kk$ be a field, $R_n=\kk[x_1,\ldots,x_n]$ a polynomial ring in $n$ variables, and $R=\bigcup_{n\ge 1} R_n=\kk[x_1,x_2,\ldots]$ be the polynomial ring in infinitely many variables. Denote by $\N$ the set of positive integers, and $\Inc$ the monoid of strictly increasing functions $\pi: \N \to \N$, with the operation being composition of maps. Then $\Inc$ acts on $R$ by endomorphisms via the rule
\[
\pi(x_i)=x_{\pi(i)}, \quad \text{for each $i\in \N$}.
\]
A family of ideals $\Icc=(I_n)_{n\ge 1}$, where $I_n\subseteq R_n$, is called an \emph{$\Inc$-invariant chain} if for all $1\le m < n$ and for all $\pi \in \Inc$ such that $\pi(m)\le n$, we have $\pi(I_m)\subseteq I_n$. A related notion is that of $\Inc$-invariant ideals: An ideal $I\subseteq R$ is called $\Inc$-invariant if $\pi(I)\subseteq I$ for any $\pi \in \Inc$. Recently, there was considerable interest in $\Inc$-invariant chains and $\Inc$-invariant ideals, due to several reasons. First, $\Inc$-invariant chains and ideals appear in various places. For instance, let $\Sym(n)$ be the symmetric group on $n$ objects and $\Sym=\bigcup_{n\ge 1} \Sym(n)$ the direct limit of these groups. Then any $\Sym$-invariant ideal $I\subseteq R$ is also $\Inc$-invariant \cite{NR17}. Let $\Icc=(I_n)_{n\ge 1}$ be any $\Sym$-invariant chain, i.e.\,each $I_n$ is $\Sym(n)$-invariant and $\pi(I_m)\subseteq I_n$ whenever $1\le m < n, \pi \in \Sym(n)$. Then any such $\Sym$-invariant chain is also $\Inc$-invariant. Special instances of $\Sym$-invariant ideals are abound in representation theory \cite{CEF15, SS17} and algebraic statistics \cite{AH07, HS12}. Second, $\Inc$-invariant chains and ideals enjoy remarkable finiteness properties due to their highly symmetrical nature. By work of Cohen \cite{Co87}, Aschenbrenner--Hillar \cite{AH07}, Hillar--Sullivant \cite{HS12}, Nagel--R\"omer \cite{NR17}, it is known that $R$ is $\Inc$-noetherian, in particular any $\Inc$-invariant ideal $I\subseteq R$ is generated by finitely many $\Inc$-orbits. Moreover, any $\Inc$-invariant chain of ideals stabilizes, namely for integers $1\le m\le n$, let 
$$
\Inc_{m,n}=\{\pi\in \Inc \mid \pi(m)\le n\},
$$
then for any $\Inc$-invariant chain $\Icc=(I_n)_{n\ge 1}$, there exists some finite $r\ge 1$ such that $I_n =\Inc_{r,n}(I_r)R_n$ for all $n\ge r$. We call the smallest such $r$ the \emph{stability index} of $\Icc$, denoted $\ind(\Icc)$.

The present work can be seen as the continuation of an active program aiming at studying asymptotic properties of $\Inc$- and $\Sym$-invariant chains of ideals; see, for example \cite{DEF,GN18,KLR22, LNNR2, MR22,NR17} for various aspects of this program.   In \cite[Conjecture 4.12]{LNNR2}, Le, Nagel, the second author, and R\"omer proposed 
\begin{conj}
\label{conj_regularity}
Let $\Icc=(I_n)_{n\ge 1}$ be an $\Inc$-invariant chain of homogeneous ideals. Then for all $n\gg 0$, $\reg I_n=Cn+D$ for certain constants $C$ and $D$ depending only on $\Icc$. Here $\reg (-)$ denotes the Castelnuovo--Mumford regularity.
\end{conj}
We say that $\Icc=(I_n)_{n\ge 1}$ is a \emph{saturated} chain if $I_n=I_{n+1}\cap R_n$ for all $n\ge 1$. In \cite{LNNR2}, the conjecture was partially solved in some cases, notably when (the quotient ring defined by) $I_n$ is artinian for all $n$, or when $\Icc$ is a saturated chain, and $I_n$ is a squarefree monomial ideal for all $n$. The main result of \cite{LN22} implies that the limit $\lim\limits_{n\to \infty} (\reg I_n /n)$ always exists for an arbitrary invariant chain of monomial ideals, supporting the validity of \Cref{conj_regularity}. The conjecture was verified by Murai \cite{Mu20} and Raicu \cite{Ra21} (using different approaches) if $\Icc=(I_n)_{n\ge 1}$ is an $\Sym$-invariant chain of monomial ideals. On the other hand, \Cref{conj_regularity} remains open when the chain $\Icc$ is not saturated, even if each $I_n$ is a squarefree monomial ideal. 

As usual, an ideal $I_n\subseteq R_n$ is called an \emph{edge ideal} if it is generated by squarefree quadratic monomials. Related to the theory of invariant chains of edge ideals, algebraic and combinatorial aspects of FI-graphs were studied in \cite{Ram20, RW19}. The goal of this paper is to prove the following results.
\begin{thm}[Theorem \ref{thm_convergence}]
\label{thm_main}
Let $\Icc=(I_n)_{n\ge 1}$ be an $\Inc$-invariant \textup{(}not necessarily saturated\textup{)} chain of eventually nonzero edge ideals. Let $r=\ind(\Icc)\ge 1$. Then for all $n\ge N=N(r)$, there is an equality $\reg I_n = \reg I_{n+1}$, and moreover $\reg I_n \in \{2,3\}$.
\end{thm}
It follows from \cite[Corollary 4.8]{LNNR2} that if $\Icc=(I_n)_{n\ge 1}$ is an $\Inc$-invariant chain of eventually nonzero squarefree monomial ideals, then the sequence $(\reg I_n)_{n\ge 1}$ is  bounded by a constant. On the other hand, there was no information about this constant. \Cref{thm_main} provides the surprising information that for $\Inc$-invariant chain of nonzero edge ideals, not only does the limit $\lim\limits_{n\to \infty} \reg I_n$ exist, but it can only be either 2 or 3. We are able to supply an explicit upper bound for the position when the sequence $(\reg I_n)_{n\ge 1}$ becomes constant, namely the number $N(r)$ in \Cref{thm_main} is at most $2(r^2+5r)$ (see \Cref{thm_convergence} and \Cref{rem_Nr}). In addition, we provide an effective characterization when $\reg I_n=2$ for all large enough $n$.

\begin{thm}[Theorem \ref{thm_convergence}]
\label{thm_main2}
Keep using notations of \Cref{thm_main}. Let $G_n$ be the graph on the vertex set $[n]=\{1,2,\ldots,n\}$ corresponding to $I_n$ for each $n$, and let $E(G_r)=\{\{i_1,j_1\},\ldots,\{i_s,j_s\}\}$ where we assume that $1\le i_t<j_t \le r,\ i_1\le i_2\le \cdots\le i_s$ and if $i_t=i_{t+1}$ then $j_t<j_{t+1}$. Denote $q=\max\{1\le t\le s: i_t=i_1\}$. Denote by $\indmatch(G_n)$ the induced matching number of the graph $G_n$. Then the following statements are equivalent:
\begin{enumerate}[\quad \rm (1)]
 \item $\reg I_n=2$ for all $n\gg 0$;
 \item  Either $j_q = \max\{j_1,\ldots,j_s\}$, or it holds that $\min \{j_t-i_t: t=1,\ldots,s\}=1$ and $\indmatch(G_{3r})=1$.
\end{enumerate}
\end{thm}
The proofs of Theorems \ref{thm_main} and \ref{thm_main2} are combinatorial, and reveal unexpected properties of the graphs $G_n$. A fairly simple but crucial observation is \Cref{lem_integralpoints}, which translates questions about the chain $(I_n)_{n\ge 1}$ and the graphs $G_n$ to questions about integral points inside certain right isosceles triangles on the coordinate plane. Most of the results in this paper were found thanks to the intuitive approach provided by this observation. Our main tools in proving Theorems \ref{thm_main} and \ref{thm_main2}  are two key Lemmas \ref{lem_twopoints_corners} and \ref{lem_twopoints_induced2K2}, which thanks to \Cref{lem_integralpoints}, are statements about induced $2K_2$ (i.e. two disjoint edges) subgraphs of $G_n$ when $n$ is large. In particular, we show in \Cref{thm_no3K2} that for all $n\ge 3r$, $G_n$ has induced matching number  $\indmatch(G_n) \in \{1,2\}$ and the equality $\indmatch(G_n)=\indmatch(G_{n+1})$ holds. 

There was a lot of work on the algebraic and combinatorial properties of edge ideals; see, for example, \cite{MV12} and the references therein. The proofs of our main results also use some known results on this topic, e.g. the celebrated Fr\"oberg theorem characterizing edge ideals with a 2-linear resolution \cite{Fr90}. Nevertheless, our work is more than just routine applications of the existing literature. For example, besides Lemmas \ref{lem_twopoints_corners} and \ref{lem_twopoints_induced2K2} on induced $2K_2$ subgraphs of $G_n$, a novelty of our approach is given by \Cref{constr_inducedanticycle}, where we construct long induced \emph{anticycles} (i.e. complementary graphs of cycles) of $G_n$ for a large family of chains. It is also worth mentioning that to prove the inequality $\reg I_n\le 3$ for $n\gg 0$ in \Cref{thm_main}, we show that $G_n$ possesses certain ``cochordal filtration'', see the proof of \Cref{thm_reg<=3}. Related to this last point, we note that the well-known inequality $\reg I_n \ge 1+\indmatch(G_n)$ is strict in general for an $\Inc$-invariant chain of edge ideals $(I_n)_{n\ge 1}$.  To summarize, our results suggest that there is an interesting interplay between the algebraic properties of $I_n$ and the combinatorics of $G_n$. We hope that our main results and techniques may shed further light on \Cref{conj_regularity} and on the combinatorics of $\Inc$-invariant chains of monomial ideals in general. Based on experiments with Macaulay2 \cite{GS96}, we are tempted to propose the following 
\begin{conj}
Let $\Icc=(I_n)_{n\ge 1}$ be an $\Inc$-invariant chain of eventually nonzero squarefree monomial ideals with $r=\ind(\Icc)$, and let $d=d(I_r)$ be the maximal generating degree of $I_r$. Then $\reg I_n \le 2d-1$ for all $n\gg 0.$
\end{conj}

\textbf{Organization}. We discuss some basic results and constructions in \Cref{sect_prelim}. Within this part, \Cref{subsect_saturated} illustrates our main results in the special cases of saturated chains. Sections \ref{sect_inducedmatchings}, \ref{sect_inducedanticyc}, and \ref{sect_longanticycles} form the technical core of this work.  In \Cref{sect_inducedmatchings}, we prove that for large $n$, the graph $G_n$ corresponding to the ideal $I_n$ in any $\Inc$-invariant chain of edge ideals has no induced $3K_2$ subgraph, and its induced matching number is a constant (either $1$ or $2$). In \Cref{sect_inducedanticyc}, we show that while the complementary graph $G_n^c$ may contain induced cycles (\Cref{ex_not_weaklychordal}), the length of such cycles has to be large (\Cref{prop_noCmc}). In \Cref{sect_longanticycles}, we construct long induced anticycles of the graph $G_n$ under some extra assumptions, see \Cref{prop_reg3_specialcase1} and particularly \Cref{constr_inducedanticycle}. This section is technically the most challenging part of the manuscript. On the other hand, in essence \Cref{sect_longanticycles} does not depend on the remaining parts,  and \Cref{prop_reg3_specialcase1} is the only result in this section that matters later on. For a quick first reading, the reader may simply take \Cref{prop_reg3_specialcase1} for granted and proceed directly to \Cref{sect_sharpbound}.  After the preparations in Sections \ref{sect_inducedmatchings}--\ref{sect_longanticycles}, we derive the more transparent and interesting results on $\Inc$-invariant chains of edge ideals in the rest of this paper.  In  \Cref{sect_sharpbound}, we prove that under the hypotheses of \Cref{thm_main}, $\reg I_n\le 3$ for all large enough $n$. We then deduce our main results, Theorems \ref{thm_main} and \ref{thm_main2}, in \Cref{sect_converge}. 

\section{Preliminaries}
\label{sect_prelim}

\subsection{Equivariant chains of ideals}
Let $R_n=\kk[x_1,\ldots,x_n]$ be a standard graded noetherian polynomial ring where $n\ge 1$, and let $R_0=\kk$. Denote by $G(J)$ the set of minimal monomial generators of a monomial ideal $J$ of $R_n$.
\begin{defn}[Maximal index in the support]
For a monomial $m\in R_n$, denote by $\Msupp(m)$ the maximal index $1\le i\le n$ such that $x_i$ divides $m$. If $m=1$, we set $\Msupp(1)=0$. For a monomial ideal $J\subseteq R_n$, define $\Msupp(J)=\max \{\Msupp(m): m\in G(J)\}$.
\end{defn}
The following slight generalization of saturated chains was introduced in \cite{LN22}.
\begin{defn}[Quasi-saturated chains]
 Let $\Icc=(I_n)_{n\ge 1}$ be an $\Inc$-invariant chain of ideals with $r=\ind(\Icc)$. Set $p_n=\Msupp(I_n)$. We say that $\Icc$ is \emph{quasi-saturated} if
 $I_{n+1}\cap R_{p_n}=I_n\cap R_{p_n}$  for all $n\ge r$.
\end{defn}

Any saturated chain is also a quasi-saturated chain. We will invoke the following simple lemma several times. Below, for each $i\ge 0$, let $\sigma_i: \N \to \N$ be the function given by 
\[
\sigma_i(j) = \begin{cases}
               j, &\text{if $1\le j\le i$},\\
               j+1, &\text{if $j\ge i+1$}.
              \end{cases}
\]
Both assertions of the next lemma are immediate from the fact that for $n\ge r=\ind(\Icc)$, it holds that $\langle \Inc_{n,n+1}(I_n) \rangle _{R_{n+1}}=I_{n+1}$.
\begin{lem}
\label{lem_maxsupp}
Let $\Icc=(I_n)_{n\ge 1}$ be an $\Inc$-invariant chain of nonzero proper monomial ideals with $r=\ind(\Icc)$. Denote $p=\Msupp I_r$, so $1\le p\le r$. Then we have $\Msupp(I_{n+1})\le\Msupp(I_n)+1$, and in particular $\Msupp I_n \le n-r+p$ for all $n\ge r$. 
\end{lem}
The following is a consequence of \cite[Lemma 4.7 and its proof]{LN22}, by setting $e=0$.
\begin{prop}
\label{prop_intersection_chain}
Let $\Icc=(I_n)_{n\ge 1}$ be an $\Inc$-invariant chain of eventually nonzero proper monomial ideals with $r=\ind(\Icc)$. Let $p=\Msupp(I_r)$. Consider collection $\Jcc=(J_n)_{n\ge 1}$ of ideals given by
$$
J_n= \begin{cases}
 \langle I_n \cap R_{n-r+p-1} \rangle_{R_n}, &\text{if $n\ge r+1$},\\
0, &\text{if $n \le r$}.
\end{cases}
$$ 
The following statements hold.
\begin{enumerate}[\quad \rm(i)]
\item $\Jcc$ is an $\Inc$-invariant chain of eventually nonzero monomial ideals and $I_n\subseteq J_{n+1}$ for all $n\ge r$.
\item There is an equality $\ind \Jcc = r+1$. 
\item The chain $\Icc$ is quasi-saturated if and only if $J_{r+1}=\langle I_r \rangle _{R_{r+1}}$.
\item There is an equality $\reg J_n= \reg (I_n,x_{n-r+p})$ for all $n\ge r+1$.
\end{enumerate} 
\end{prop}
The following invariant is useful for inductive arguments.
\begin{defn}
\label{defn_q-invariant}
Let $J\subseteq R_n$ be a nonzero proper monomial ideal such that 
$\Msupp(J)=p\ge 1$. Let $\delta(J)$ denote the maximal degree of an element of $G(J)$. Then the \emph{$q$-invariant} of $J$ is defined to be the number
\[
{q}(J) 
= 
\sum_{j=0}^{\delta(J)}\dim_\kk \Big(\frac{R_p}{J\cap R_p}\Big)_j.
\]

Let $\Icc=(I_n)_{n\ge 1}$ be an $\Inc$-invariant chain with $r=\ind(\Icc)$. We call $q(\Icc):=q(I_r)$ the \emph{$q$-invariant} of the chain $\Icc$.
\end{defn}

If $J=R_n$, we use the convention that $q(J)=0$. Hence $q(J)=0$ if and only if $J=R_n$.
\begin{ex}
Let $J=(x_1x_3,x_2x_4)\subseteq R_5=\kk[x_1,\ldots,x_5]$. Then $\delta(J)=2$, $\Msupp(J)=4$, so
\[
q(J)= \sum_{j=0}^{2}\dim_\kk \Big(\frac{R_4}{J\cap R_4}\Big)_j=1+4+8=13.
\]
\end{ex}
The next lemma is immediate.
\begin{lem}
\label{lem_q-invariant}
The following statements hold.
\begin{enumerate}[\quad \rm (1)]
\item For any monomial ideal $I\subseteq R_n$, there is an equality $q(I)=q(\langle I \rangle_{R_{n+1}})$.

\item Let $I\subseteq J$ be nonzero monomial ideals of $R_n$ such that $\delta(I)\ge \delta(J)$ and $\Msupp(I)\ge \Msupp(J)$. Then there is an inequality
\[
q(I)\ge q(J).
\]
The equality holds if and only if $I=J$. 
\end{enumerate}
\end{lem}

\subsection{Casteluovo--Mumford regularity}
Let $S=\kk[x_1,\ldots,x_n]$ be a standard graded polynomial ring where $n\ge 1$. Let $M$ be a finitely generated graded $S$-module. Then (\emph{Castelnuovo--Mumford}) \emph{regularity} of $M$ is
\[
\reg M=\sup\{j-i: \Tor^S_i(\kk,M)_j\neq 0\}.
\]
The following useful lemma on regularity of monomial ideals is from \cite[Lemma 2.10]{DHS13} and \cite[Corollary 3.3 and Theorem 4.7]{CHH+19}.
\begin{lem}
\label{lem_reg_mon}
Let $J$ be a monomial ideal of $S$ and $x$ any variable of $S$. Then there is a chain
\[
\max\{\reg(J:x), \reg(J,x)\} \le \reg J \in \{\reg(J:x)+1, \reg (J,x)\}.
\]
Here $(J,x)$ denotes the ideal $J+(x)$.
\end{lem}
\subsection{Edge ideals and graphs}
\label{subsect_edgeideals}

Let $G$ be a simple graph on the vertex set $[n]=\{1,\ldots,n\}$, and let $E(G)$ be its edge set. The \emph{edge ideal} of $G$ is the following ideal of $R_n$:
\[
I(G)=(x_ix_j \mid \{i,j\} \in E(G)).
\]
An \emph{induced subgraph} $H$ of $G$ is a graph with vertex set $V(H)\subseteq V(G)$ and edge set $E(H)\subseteq E(G)$, such that two vertices $i, j\in V(H)$ are adjacent in $H$ if and only if they are adjacent in $G$.

The complementary graph of $G$ is denoted by $G^c$. The \emph{cycle} $C_m$ is the graph on $m\ge 3$ vertices $v_1,\ldots,v_m$ with the edge set
\[
E(C_m)=\left\{\{v_1,v_2\}, \{v_2,v_3\},\ldots,\{v_m,v_1\} \right\}.
\]
We call $m$ the \emph{length} of the cycle $C_m$. An \emph{anticycle} is the complementary graph $C_m^c$ of the cycle $C_m$ for some $m\ge 4$. By abuse of terminology, we call $m$ the \emph{length} of the anticycle $C_m^c$, and say that an anticycle has \emph{consecutive vertices} $v_1,v_2,\ldots,v_m$ if its complementary graph is the cycle with edges $v_1v_2,\ldots,v_{m-1}v_m,v_mv_1$.

We say that $G$ is \emph{chordal} if it has no induced cycle of length at least 4. Dually, we say $G$ is \emph{cochordal} if $G^c$ is chordal. 

For each $g\ge 1$, the graph $gK_2$ is a union of $g$ disjoint edges. The \emph{induced matching number} of $G$ is
\[
\indmatch(G)=\max\{g \mid \text{$G$ has an induced $gK_2$ subgraph}\}.
\]
Clearly $C_4^c$ is isomorphic to $2K_2$, therefore if $G$ is a cochordal graph with at least one edge, then as $G$ does not contain any induced $C_4^c$, $\indmatch(G)=1$.

The first part of the following lemma is contained in \cite[Lemma 2.2]{Ka06} and \cite[Lemma 7]{Wo14}. The second part is contained in \cite[Theorem 1]{Fr90}.
\begin{lem}
\label{lem_reg_indmatch}
Let $G$ be a graph with at least one edge. The following statements hold.
\begin{enumerate}[\quad \rm (1)]
 \item There is an inequality $\reg I(G)\ge 1+\indmatch(G)$.
 \item \textup{(Fr\"oberg's theorem)} We have $\reg I(G)=2$ if and only if $G$ is a cochordal graph.
\end{enumerate}
\end{lem}
For a subset $U$ of the vertex set $V(G)$ of $G$, the \emph{deletion} of $G$ to $U$, denoted $G\setminus U$, has the vertex set $V(G\setminus U):=V(G)\setminus U$ and the edge set
\[
E(G\setminus U):=\{\{x,y\} \in E(G) \mid x,y \in V(G)\setminus U\}.
\]
For a vertex $v\in G$, its \emph{open neighborhood} $N_G(v)$ is the set of vertices $u\neq v$ that are adjacent to $v$, and its \emph{closed neighborhood} is
\[
N_G[v]:= N_G(v) \cup \{v\}.
\]
Let $G$ be a graph on $[n]$, $v \in V(G)$ a vertex. Let $S=\kk[x_1,\ldots,x_n]$, $J$ a monomial ideal of $S$, and $y$ a linear form of $S$. We have a standard exact sequence
\[
0\to \left(\frac{S}{J:y}\right)(-1) \to \frac{S}{J} \to \frac{S}{(J,y)} \to 0.
\]
Using this sequence for $J=I(G)$ and $y=x_v$, we may prove the following statement.
\begin{lem}[{\cite[Lemma 3.1]{DHS13}}]
\label{lem_reg_edgeideal_deletion}
There is an inequality
\[
\reg I(G) \le \max\{\reg I(G\setminus N_G[v])+1, \reg I(G\setminus v)\}.
\]
\end{lem}

\subsection{Regularity of saturated chains}
\label{subsect_saturated}
The following result together with its simple proof serve as a motivation for Theorems \ref{thm_main} and \ref{thm_main2}.
\begin{lem}
\label{lem_reg_quasi-saturated}
 Let $\Icc=(I_n)_{n\ge 1}$ be a quasi-saturated chain of eventually nonzero edge ideals with $\ind(\Icc)=r$. Then we have an equality $\reg I_n=2$ for all $n\ge r$, i.e. $I_n$ has a $2$-linear resolution for all such $n$.
\end{lem}
\begin{proof}
By Fr\"oberg's theorem (\Cref{lem_reg_indmatch}(2)), we have to show that the graph $G_n$ corresponding to $I_n$ is cochordal for all $n\ge r$.

Assume the contrary, that $i_1,\ldots,i_m$ is an induced cycle of length $m\ge 4$ in $G_n^c$ for some $n\ge r$. By reindexing, we may assume that $i_1=\min\{i_j: 1\le j\le m\}$. Note that $i_1,\ldots,i_m$ are pairwise distinct.

We claim that $i_2 < \min\{i_3,\ldots,i_{m-1}\}$. If this is not the case, then $i_2 > i_j$ for some $3\le j\le m-1$. Since $\{i_1,i_j\}\notin G_n^c$, $x_{i_1}x_{i_j}\in I_n$, and the $\Inc$-invariance of $\Icc$ implies that $x_{i_1}x_{i_2}\in I_{n+i_2-i_j}$. Let $p=\Msupp(I_n)$. The quasi-saturatedness of $\Icc$ implies that $I_s\cap R_p \subseteq I_n$ for all $s\ge n$. Now $x_{i_1}x_{i_2} \in I_{n+i_2-i_j} \cap R_p \subseteq I_n$, hence $\{i_1,i_2\}\in G_n$, a contradiction. Similarly $i_m < \min\{i_3,\ldots,i_{m-1}\}$.

Hence $\min\{i_2,i_3,\ldots,i_m\}=\min\{i_2,i_m\}$. We assume that $\min\{i_2,i_3,\ldots,i_m\}=i_2$, the remaining case being similar. Now $\{i_2,i_m\}\notin G_n^c$, so $x_{i_2}x_{i_m}\in I_n$. From above $i_m<i_3$, so arguing as in the last paragraph, we deduce $x_{i_2}x_{i_3}\in I_n$. This means $\{i_2,i_3\} \in G_n$, a contradiction. Hence $G_n$ is cochordal for all $n\ge r$. The proof is completed.
\end{proof}
\begin{ex}
The behaviour of the regularity for non-quasi-saturated chains of edge ideals is much more delicate. For instance, let $(I_n)_{n\ge 1}$ be the $\Inc$-invariant chain with $I_n=(0)$ for $n\le 9$, $I_{10}=(x_1x_{10},x_2x_4,x_3x_5,x_7x_9)$ and stability index $r=10$. Computations with Macaulay2 \cite{GS96} give the following data.
 \begin{table}[ht!]
  \begin{tabular}{|c | c | c | c | c | c | c | c |c | c | c |}
  \hline
  $n$ & 10 & 11 & 12 & 13 & 14 & 15 & 16 & 17 & 18 & 19 \\
  \hline
  $\reg I_n$ & 5 & 4 & 3 & 4 & 4 & 3 & 3 & 3 & 3 & 2 \\
  \hline
  \end{tabular}
 \end{table}
 
In particular, the sequence $(\reg I_n)_{n\ge 1}$ might be fairly irregular at the beginning. By \Cref{prop_reg2_all-embracinginterval}, we have $\reg I_n=2$ for all $n\ge 3r=30$.
\end{ex}

\section{Induced matchings}
\label{sect_inducedmatchings}
The goal of this section is to prove
\begin{thm}
\label{thm_no3K2}
Let $\Icc=(I_n)_{n\ge 1}$ be an $\Inc$-invariant chain of eventually nonzero edge ideals with the stability index $\ind \Icc=r\ge 1$. Let $G_n$ be the graph corresponding to $I_n$. Then for all $n\ge 3r$, the graph $G_n$ does not contain any induced $3K_2$ subgraph. In other words, we have $\indmatch(G_n)\le 2$ for all $n\ge 3r$. 

Furthermore, $\indmatch(G_n)=\indmatch(G_{n+1})$ for all $n\ge 3r$.
\end{thm}

The main work is done by two key lemmas \ref{lem_twopoints_corners} and \ref{lem_twopoints_induced2K2}. In the sequel, we use the following notations. 
\begin{notn}
For a point $(i,j)\in \R^2$ and a positive real number $n$, denote by $\Delta((i,j),n)$ the right isosceles triangles with the vertices $(i,j), (i,j+n), (i+n,j+n)$. 
\end{notn}

For example, \Cref{fig_triangles} depicts the triangles $\Delta((2,7),2)$ and $\Delta((3,4),2)$. Questions about $\Inc$-invariant chains of edge ideals can be translated into questions about integral points in the triangles $\Delta((i,j),n)$ via the following simple but crucial observation. Below, for $x,y\in \R$, denote by $(x,y)^{\le}$ the point $(\min(x,y),\max(x,y))\in \R^2$.

\begin{lem}
\label{lem_integralpoints}
Let $1\le i< j\le r$ and $n\ge r$ be positive integers. Then for integers $u<v$, the following are equivalent:
\begin{enumerate}[\quad \rm (1)]
 \item $x_ux_v\in \Inc_{r,n}(x_ix_j)$;
 \item $(u,v)\in \Delta((i,j),n-r)$.
\end{enumerate}
In particular, let $\Icc=(I_n)_{n\ge 1}$ be an $\Inc$-invariant chain of eventually nonzero edge ideals with $r=\ind(\Icc)$, and denote $E(G_r)=\{\{i_1,j_1\},\ldots,\{i_s,j_s\}\}$ where $1\le i_p < j_p\le r$ for $p=1,\ldots,s$. Then for each $n\ge 0$,
\[
E(G_{n+r})=\bigcup_{p=1}^s \left\{\{u,v\}: u,v\in \Z, (u,v)^\le \in \Delta((i_p,j_p),n) \right\}.
\]
\end{lem}
\begin{proof}
The second assertion follows from the fact that $I_{n+r}=\langle \Inc_{r,n+r}(I_r) \rangle_{R_{n+r}}$ and the first assertion. Let us prove the latter.

 (1) $\Rightarrow$ (2): Since $x_ux_v\in \Inc_{r,n}(x_ix_j)$, for some $\pi\in \Inc_{r,n}$, we have $u=\pi(i), v=\pi(j)$. Since $\pi$ is increasing, $\pi(r)\le n$, and $1\le i<j\le r$, we have
 \[
0\le \pi(i)-i \le \pi(j)-j \le \pi(r)-r \le n-r.
 \]
 This implies $(u,v)\in \Delta((i,j),n-r)$.

  \begin{center}
 \begin{figure}[ht!]
  \includegraphics[width=32ex]{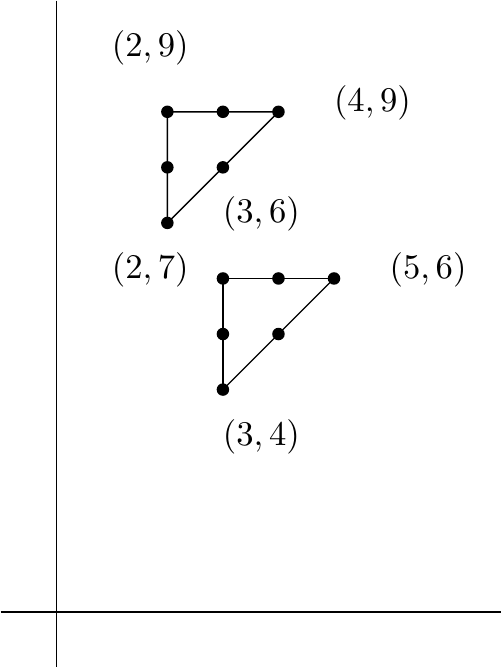} 
  \caption{An invariant chain with $I_7=(x_2x_7,x_3x_4)$}
  \label{fig_triangles}
 \end{figure} 
 \end{center}
 
 (2) $\Rightarrow$ (1): Assume that $(u,v)\in \Delta((i,j),n-r)$. Then there is a chain
 \[
0\le u-i\le v-j\le n-r.
 \]
Choose $\pi: \N \to \N$ such that 
$$\pi(t)= \begin{cases}
           t+u-i, &\text{if $t\le j-1$},\\
           t+v-j, &\text{if $t\ge j$}.
       \end{cases}
$$
Then $\pi \in \Inc_{r,n}$, as $\pi(r)=r+v-j\le n$. It is also clear that $\pi(i)=u, \pi(j)=v$, hence $x_ux_v=\pi(x_ix_j)\in \Inc_{r,n}(x_ix_j)$, as desired.
\end{proof}
\begin{ex}
Let $\Icc$ be an $\Inc$-invariant chain with $I_7=(x_2x_7,x_3x_4)$ and $I_n=\Inc_{7,n}(I_7)$ for $n\ge 7$. Then using \Cref{lem_integralpoints}, we can compute
\begin{align*}
&I_9= \langle \Inc_{7,9}(x_2x_7) \rangle_{R_9}+ \langle \Inc_{7,9}(x_3x_4) \rangle_{R_9}\\
&= (x_2x_7,x_2x_8,x_2x_9,x_3x_8,x_3x_9,x_4x_9)+(x_3x_4,x_3x_5,x_3x_6,x_4x_5,x_4x_6,x_5x_6). 
\end{align*}
See Figure \ref{fig_triangles}. More generally, for each $n>7$, the generators of the ideal $I_n$ correspond to the lattices points in the union of the two triangles $\Delta((2,7),n-7)$ and $\Delta((3,4),n-7)$.
\end{ex}

Let $A_1=(i_1,j_1), A_2=(i_2,j_2)$, $A_3=(i_3,j_3)\in \R_{\ge 1}^2$ be three given points such that $i_p< j_p$. For each $n$, let $\Delta_{p,n}=\Delta((i_p,j_p),n)$.

The first key lemma in this section shows that if  $A_1=(i_1,j_1), A_2=(i_2,j_2) \in \R_{\ge 1}^2$ are such that $i_p< j_p, i_1\le i_2$, $n$ is sufficiently large, and $(u_p,v_p)\in \Delta_{p,n}$, $p=1,2$ are such that the rectangle with vertices $(u_i,v_j)$, $1\le i,j\le 2$ has exactly two vertices in $\bigcup_{p=1}^2\Delta_{p,n}$ (namely $(u_1,v_1)$ and $(u_2,v_2)$), then $u_1$ has to be small, while $v_2$ has to be large. In particular, $(u_1,v_1)$ has to lie to the left of $\Delta_{2,n}$, while $(u_2,v_2)$ has to lie above $\Delta_{1,n}$ (see Figure \ref{fig_twopoints}).

\begin{lem}
\label{lem_twopoints_corners}
Assume that $A_1=(i_1,j_1), A_2=(i_2,j_2) \in \R_{\ge 1}^2$ are \textup{(}not necessarily distinct\textup{)} points such that $i_p< j_p, i_1\le i_2$. Let $n\ge \max\{2(j_1-i_1+i_2-j_2),1\}$ be an integer. Let $(u_p,v_p)\in \Delta_{p,n}$, $p=1,2$ be two points such that $(u_1,v_2), (u_2,v_1) \notin \bigcup_{p=1}^2\Delta_{p,n}$. Then the following inequalities hold:
\begin{enumerate}[\quad \rm (1)]
 \item $u_1< v_1, u_2<v_2$;
 \item $u_1<i_2$ and $v_2> j_1+n$;
 \item $i_1<i_2$ and $j_1<j_2$, thus $A_1$ and $A_2$ are distinct.
\end{enumerate}
\end{lem}
\begin{proof}

Denote $\Delta_p=\Delta_{p,n}$ for short. Note that 
$$
\Delta_p=\{(u,v)\in \R^2: 0\le u-i_p \le v-j_p \le n\}.
$$

For (1): Since $v_p\ge u_p+j_p-i_p$ and $j_p>i_p$, we get $v_p>u_p$, namely (1) is true.

For (2): By the  hypothesis $(u_2,v_1)\notin \Delta_1$, so the following chain is invalid 
\[
0\le u_2-i_1 \le v_1-j_1\le n.
\] 
But $u_2\ge i_2$, so $u_2-i_1\ge i_2-i_1\ge 0$, namely the first inequality of the last chain is valid. The last inequality is also valid since $(u_1,v_1)\in \Delta_1$. Hence the middle one is false, namely 
$$
u_2-i_1 > v_1-j_1 \ge u_1-i_1.
$$
The last inequality follows from $(u_1,v_1)\in \Delta_1$. From the last chain, we get $u_1 < u_2$.

 \begin{center}
 \begin{figure}[ht!]
  \includegraphics[width=70ex]{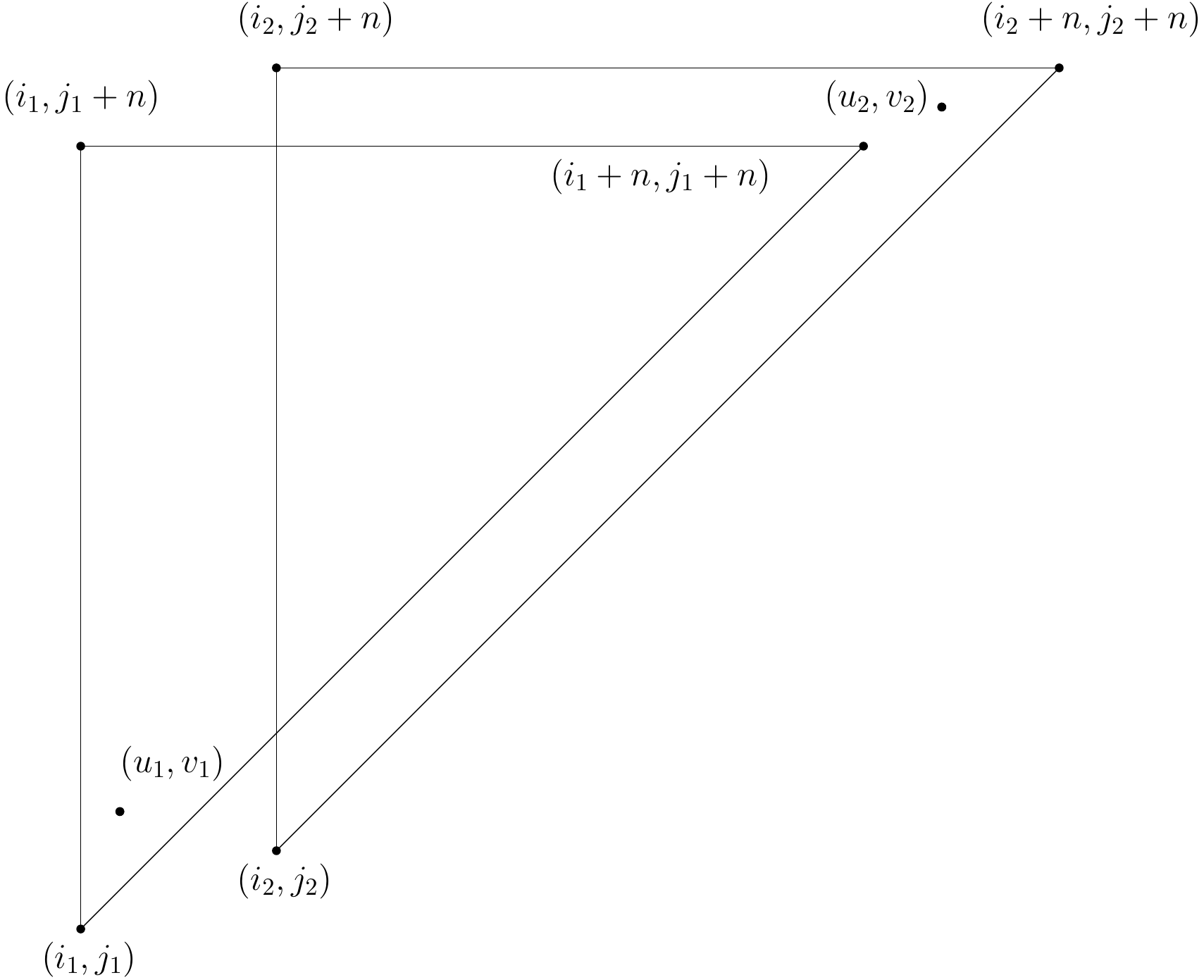} 
  \caption{Two points $(u_1,v_1), (u_2,v_2)$ as in \Cref{lem_twopoints_corners}}
  \label{fig_twopoints}
 \end{figure} 
 \end{center}

By the hypothesis $(u_1,v_2)\notin \Delta_2$, so the following chain is invalid
\[
0\le u_1-i_2  \le v_2-j_2 \le n.
\]
The middle and the last inequalities are true since from $(u_2,v_2)\in \Delta_2$, and $u_1<u_2$, we deduce $u_1-i_2 < u_2-i_2 \le v_2-j_2 \le n$.
Hence the first inequality of the last display is false, namely $u_1 <i_2$. 

Next we have to show that $v_2 > n+j_1$. Assume the contrary that $v_2\le n+j_1$. By the hypothesis, $(u_1,v_2)\notin \Delta_1$, hence the following chain is invalid
\begin{equation}
\label{eq_ineq_invalidchain_u1v2}
0\le u_1-i_1\le v_2-j_1\le n.
\end{equation}
The first inequality holds. By assumption, the last inequality holds as well. Therefore the middle inequality does not hold, namely $v_2 < u_1-i_1+j_1 \le v_1$, using  $(u_1,v_1)\in \Delta_1$. Hence as $u_1<i_2$,
\begin{equation}
\label{eq_ineq_v2small}
v_2 < \min\{v_1, i_2-i_1+j_1\}.
\end{equation}

By the hypothesis $(u_2,v_1)\notin \Delta_2$, so the following chain is invalid
\[
0\le u_2-i_2 \le v_1-j_2\le n.
\]
Thanks to \eqref{eq_ineq_v2small} and $(u_2,v_2)\in \Delta_2$, $0\le u_2-i_2 \le v_2-j_2 < v_1-j_2$. Hence the first two inequalities in the last display are valid, and the last one is not. In other words,
\begin{equation}
\label{eq_ineq_v1large}
v_1 > n+j_2.
\end{equation}
 
By the hypothesis, $(u_2,v_1)\notin \Delta_1$, so the following chain is invalid
\[
0\le u_2-i_1\le v_1 -j_1 \le n.
\]
The first and the last inequalities are true, since $u_2\ge i_2\ge i_1$ and $(u_1,v_1)\in \Delta_1$. Hence the middle one is false, namely $u_2 > v_1+i_1-j_1$. Together with \eqref{eq_ineq_v1large}, this yields
\begin{equation}
\label{eq_ineq_u2large}
u_2 > n+j_2+i_1-j_1.
\end{equation}

Now $(u_2,v_2)\in \Delta_2$, so $u_2\le v_2-j_2+i_2$. Combining with \eqref{eq_ineq_u2large} and \eqref{eq_ineq_v2small}, we get
\[
n+ j_2+i_1-j_1 < u_2 \le v_2-j_2+i_2 < 2i_2-j_2-i_1+j_1.
\]
This yields
\[
n < 2(j_1-i_1+i_2-j_2),
\]
contradicting the hypothesis $n\ge \max\{2(j_1-i_1+i_2-j_2),1\}$. Hence the assumption $v_2\le n+j_1$ is wrong, and $v_2>j_1+n$, finishing the proof of (2).

For (3): We get $i_1\le u_1 <i_2$ and $j_1+n <v_2 \le j_2+n$, consequently $j_1<j_2$. Hence $A_1$ and $A_2$ are distinct. This concludes the proof.
\end{proof}

The next key lemma of this section shows that if in the situation of \Cref{lem_twopoints_corners}, fewer points coming from $(u_1,v_1)$ and $(u_2,v_2)$ are allowed to be in $\Delta_{1,n} \cup \Delta_{2,n}$, then $(u_1,v_1)$ must be in the lower-left corner of $\Delta_{1,n}$, near $(i_1,j_1)$ and $(u_2,v_2)$ must be in the upper-right corner of $\Delta_{2,n}$, near $(i_2+n,j_2+n)$. Roughly speaking, if $(I_n)_{n\ge 0}$ in an $\Inc$-invariant chain of edge ideals where $I_n=I(G_n)$, $u_1v_1$ and $u_2v_2$ form an induced $2K_2$ subgraph of $G_n$, where $n$ is large and $u_1\le u_2, u_p<v_p$ for $p=1,2$, then $v_1$ has to be small, while $u_2$ has to be approximately $n$ (see \Cref{fig_twopoints_2K2}).

\begin{lem}
\label{lem_twopoints_induced2K2}
Assume that $A_1=(i_1,j_1), A_2=(i_2,j_2) \in \R_{\ge 1}^2$ are points such that $i_p< j_p, i_1\le i_2$. Let $n\ge \max\{2(j_1-i_1+i_2-j_2),j_1+j_2-2i_1, 1\}$ be an integer. Let $(u_p,v_p)\in \Delta_{p,n}$, $p=1,2$ be two points such that none of the points $(u_1,u_2),(v_1,v_2),(u_1,v_2),(u_2,v_1)$ belongs to $\bigcup_{p=1}^2\Delta_{p,n}$. Then the following inequalities hold:
\[
v_1 < i_2 < n+j_1 < u_2.
\]
In particular, the closed intervals $[u_1,v_1]$ and $[u_2,v_2]$ are disjoint.

Conversely, let $(u_p,v_p)\in \Delta_{p,n}$, $p=1,2$ be two points such that 
\[
v_1 < i_2 < n+j_1 < u_2.
\]
Then none of the points $(u_i,u_j), (u_i,v_j), (v_j,u_i), (v_i,v_j)$ where $1\le i,j\le 2$, $i\neq j$, belongs to the set $\bigcup_{p=1}^2\Delta_{p,n}$. 

\end{lem}
\begin{proof}
Let $\Delta_p=\Delta_{p,n}$. Since $(u_1,v_2), (u_2,v_1)$ do not belong to $\Delta_1\cup \Delta_2$,  by \Cref{lem_twopoints_corners}, we have $u_1<v_1, u_2<v_2$, and
\begin{align}
u_1  &<i_2, \label{eq_u1small}\\
v_2 &>j_1+n. \label{eq_v2large}
\end{align}

\begin{center}
 \begin{figure}[ht!]
  \includegraphics[width=68ex]{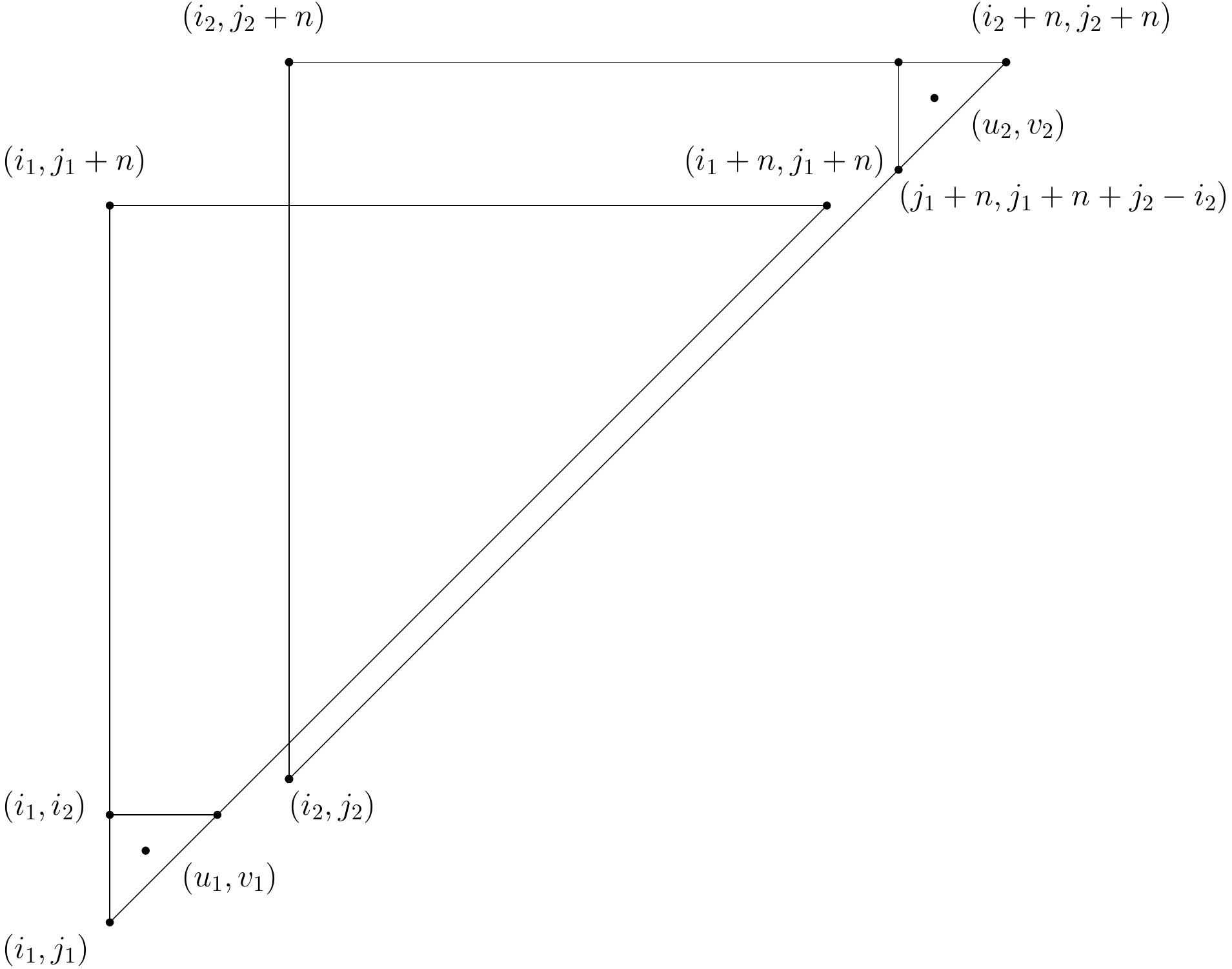} 
  \caption{Two points $(u_1,v_1), (u_2,v_2)$ as in \Cref{lem_twopoints_induced2K2}}
  \label{fig_twopoints_2K2}
 \end{figure} 
 \end{center}
 
Assume the contrary that $v_1\ge i_2$. By the assumption $(v_1,v_2)\notin \Delta_2$, so the chain $0\le v_1-i_2 \le v_2-j_2\le n$ is invalid.
Since the first and the last inequalities hold, the middle one is false. Consequently 
$$
v_1-i_2 > v_2-j_2 \ge u_2-i_2,
$$ 
using $(u_2,v_2)\in \Delta_2$. Together with \eqref{eq_v2large}, we deduce
\begin{equation}
\label{eq_ineq_v1large_lem34}
v_1 > \max\{u_2, n+j_1-j_2+i_2\}.
\end{equation}
By the hypothesis, $(u_2,v_1)\notin \Delta_1$, so the following chain is invalid
\[
0\le u_2-i_1\le v_1 -j_1 \le n.
\]
The first and the last inequalities are true, since $u_2\ge i_2\ge i_1$ and $(u_1,v_1)\in \Delta_1$. Hence the middle one is false, namely
\[
u_2 > v_1+i_1-j_1.
\]
Using \eqref{eq_ineq_v1large_lem34}, this yields
\begin{equation}
\label{eq_ineq_u2large_lem34}
u_2 > n+i_1-j_2+i_2.
\end{equation}

By the hypothesis, $(u_1,u_2)\notin \Delta_1$, so the following chain is false
\[
0\le u_1-i_1 \le u_2-j_1 \le n.
\]
Since $(u_1,v_1)\in \Delta_1$, the first inequality in the chain is valid. Per \eqref{eq_ineq_v1large_lem34}, $u_2< v_1$, so $u_2-j_1< v_1-j_1\le n$. Hence in the last chain only the middle inequality can be false, and consequently
\[
u_1 > u_2-j_1+i_1 > n+2i_1-j_1+i_2-j_2.
\]
The second inequality follows from \eqref{eq_ineq_u2large_lem34}. 

Per \eqref{eq_u1small}, $n+2i_1-j_1+i_2-j_2 < u_1<i_2$, so $n < j_2+j_1-2i_1$, contradicting the hypothesis on $n$. Therefore the assumption $v_1\ge i_2$ is wrong, and $v_1 < i_2$.

The last inequality implies $v_1 < i_2\le u_2$. Now $0\le u_1-i_1 \le v_1-j_1 < u_2-j_1$, but $(u_1,u_2)\notin \Delta_1$, so the chain $0\le u_1-i_1 \le u_2-j_1 \le n$
is invalid. Only the last inequality can be wrong, so $u_2 > n+j_1$, as desired.

Finally, for the remaining inequality $i_2 < n+j_1$, we use
\[
n+j_1 \ge (j_1+j_2-2i_1)+j_1 =j_2+2(j_1-i_1) > i_2,
\]
as $j_p>i_p$. Hence $v_1 < i_2 < n+j_1 < u_2$.

The second assertion holds since $u_1 <v_1 <u_2<v_2$.

The third assertion is a simple accounting. Indeed, as $i_p<j_p$, points in $\bigcup_{p=1}^2\Delta_{p,n}$ are of the form $(x,y)$ with $x<y$. Hence using the hypothesis 
\[
v_1 < i_2 < n+j_1 < u_2,
\]
we only need to consider the points $(u_1,u_2), (u_1,v_2), (v_1,u_2), (v_1,v_2)$. Each of these points does not lie in $\Delta_{2,n}$ since $u_1 < v_1 <i_2$. Also none of them lies in $\Delta_{1,n}$ since $n+j_1 < u_2 < v_2$. The proof is completed.
\end{proof}
\begin{cor}
\label{cor_no3K2_explicit}
Let $A_1=(i_1,j_1), A_2=(i_2,j_2), A_3=(i_3,j_3)\in \R_{\ge 1}^2$ be three given points such that $i_p< j_p$. Assume that $i_1\le i_2\le i_3$.

Denote 
\begin{gather*}
 N=\max \{2(j_1-i_1+i_2-j_2),2(j_1-i_1+i_3-j_3),2(j_2-i_2+i_3-j_3),\\
      \qquad j_1+j_2-2i_1,j_1+j_3-2i_1,j_2+j_3-2i_2,1)\}. 
\end{gather*}
Then $N \le M=2\max\{j_1,j_2,j_3\}$. Moreover, for all $n\ge M$ and for every three points $(u_p,v_p)\in \Delta_{p,n}$, $p=1,2,3$, one of the twelve points $(u_i,u_j), (v_i,v_j)$, $(u_s,v_t)$ \textup{(}where $1\le i<j\le 3, 1\le s, t\le 3, s\neq t$\textup{)}, belongs to $\bigcup_{p=1}^3\Delta_{p,n}$.
\end{cor}
\begin{proof}
Since $1\le i_p<j_p$, we get $N \le 2\max\{j_1,j_2,j_3\}$. Assume the contrary, that for some $n\ge 2\max\{j_1,j_2,j_3\}$, and some three points $(u_p,v_p)\in \Delta_{p,n}$, $p=1,2,3$, none of the twelve points $(u_i,u_j), (v_i,v_j)$, $(u_s,v_t)$ (where $1\le i<j\le 3, 1\le s, t\le 3, s\neq t$), belongs to $\bigcup_{p=1}^3\Delta_{p,n}$.

Using Lemma \ref{lem_twopoints_induced2K2} for  $(A_1, A_2, (u_1,v_1), (u_2,v_2))$ and $(A_2,A_3, (u_2,v_2), (u_3,v_3))$, we get the following inequalities 
\begin{align*}
& u_2 > n+j_1 \\
& u_2\le v_2 <i_3.
\end{align*}
But then $n < i_3 < \max\{j_1,j_2,j_3\}$. This contradiction finishes the proof.
\end{proof}

Now we are ready for the
\begin{proof}[Proof of \Cref{thm_no3K2}]
Assume the contrary, for some $n\ge 3r$, $G_n$ contains an induced $3K_2$ subgraph, say $\{u_1v_1,u_2v_2,u_3v_3\}$, where $u_p<v_p$. Since $r=\ind \Icc$, there exist $1\le i_p<j_p\le r, p=1,2,3$ such that $x_{i_p}x_{j_p}\in I_r$ and $x_{u_p}x_{v_p}$ belongs to the $\Inc$-orbit of $x_{i_p}x_{j_p}$. Again denote $\Delta_{p,n-r}=\Delta((i_p,j_p),n-r)$. By \Cref{lem_integralpoints}, that $x_{u_p}x_{v_p} \in \Inc_{r,n}(x_{i_p}x_{j_p})$ yields $(u_p,v_p)\in \Delta_{p,n-r}$. We may assume $i_1\le i_2\le i_3$. 

Now as $n\ge 3r$, $n-r \ge 2r\ge  2\max\{j_1,j_2,j_3\}$, hence by \Cref{cor_no3K2_explicit}, the $3K_2$ graph $\{u_1v_1,u_2v_2,u_3v_3\}$ is not an induced subgraph of $G_n$. This contradiction shows that $G_n$ does not contain any induced $3K_2$ subgraph for all $n\ge 3r$.  

 Take $n\ge 2r$. We prove that $\indmatch(G_{n+r})=\indmatch(G_{n+r+1})$.

\textsf{Step 1}: First assuming $\indmatch(G_{n+r})=2$ and $n\ge 2r$, we show the equality $\indmatch(G_{n+r+1})=2$. Since $\indmatch(G_{n+r})=2$ there exists an induced $2K_2$ in $G_{n+r}$, say $u_1v_1, u_2v_2$ where $u_p < v_p$. We may assume that $u_1 < u_2$. Assume that $\{i_1,j_1\}, \{i_2,j_2\} \in G_r$ where $1\le i_p < j_p\le r$ and $(u_p,v_p)\in \Delta((i_p,j_p),n)$. Since $n\ge 2r$, by \Cref{lem_twopoints_induced2K2}, we have
\begin{align}
\label{ineq_vertices}
v_1 < i_2 < n+j_1 < u_2.
\end{align}
Assume that $E(G_r)=\{\{i_1,j_1\},\ldots,\{i_s,j_s\}\}$ where $1\le i_p < j_p\le r$ for $p=1,\ldots,s$.

Now $u_1v_1, u_2v_2$ is an induced $2K_2$ in $G_{n+r}$, hence we have 
\begin{equation}
\label{eq_noedges}
\text{$(u_1,u_2), (u_1,v_2),(v_1,u_2), (v_1,v_2) \notin \Delta((i_p,j_p),n)$ for any $1\le p\le s$.}
\end{equation}
\textbf{Claim 1}: The conditions \eqref{eq_noedges} imply that $u_1v_1, (u_2+1)(v_2+1)$ form a $2K_2$ in $G_{n+r+1}$.

Assume the contrary, then from \Cref{lem_integralpoints}, for some $1\le p\le s$, one of the points $(u_1,u_2+1), (u_1,v_2+1),(v_1,u_2+1), (v_1,v_2+1)$ belongs to $\Delta((i_p,j_p),n+1)$. We treat the case $(u_1,u_2+1) \in \Delta((i_p,j_p),n+1)$, the remaining cases are similar. The following chain is valid
\[
0\le u_1-i_p\le u_2+1-j_p\le n+1.
\]
In particular, $i_p\le u_1, u_2\le n+j_p$. Per \eqref{ineq_vertices} and $n\ge 2r$, we also have 
\[
u_2-u_1 \ge u_2-v_1 > n+j_1-i_2 \ge 2r+j_1-i_2 \ge r > j_p-i_p.
\]
Combining these inequalities yields $(u_1,u_2) \in \Delta((i_p,j_p),n)$, contradicting \eqref{eq_noedges}. Hence $u_1v_1, (u_2+1)(v_2+1)$ form a $2K_2$ in $G_{n+r+1}$, whence the desired conclusion $\indmatch(G_{n+r+1})=2$.

\textsf{Step 2:} Next we show that if $\indmatch(G_{n+r})=2$ and $n\ge 2r+1$, then also $\indmatch(G_{n+r-1})=2$.

Choose $u_1v_1, u_2v_2, i_1,j_1,i_2,j_2$ as in Step 1. We claim that $u_1v_1,(u_2-1)(v_2-1)$ form a $2K_2$ in $G_{n+r-1}$.

Since $(u_1,v_1)\in \Delta((i_1,j_1),n)$, we get $(u_1,v_1)\in \Delta((i_1,j_1),v_1-j_1)$. But by \eqref{ineq_vertices}, $v_1-j_1 < i_2-j_1 < r-1$, so as $n\ge 2r+1$, $(u_1,v_1)\in \Delta((i_1,j_1),n-1)$, namely $u_1v_1\in E(G_{n+r-1})$. As per \eqref{ineq_vertices}, 
$$
u_2-1 > n+j_1-1 \ge 2r+j_1 > i_2,
$$
we deduce $(u_2-1,v_2-1)\in \Delta((i_2,j_2),n-1)$, so that $(u_2-1)(v_2-1)\in E(G_{n+r-1})$. Since $n\ge 2r+1$ and $i_2\le r$, together with \eqref{ineq_vertices},
\[
v_1 < i_2 < n-1+j_1 <u_2-1.
\]
Assume now that $u_1v_1$ and $(u_2-1)(v_2-1)$ do not form a $2K_2$ in $G_{n+r-1}$. Then as $u_1 < v_1 < u_2-1< v_2-1$, for some $1\le p\le s$, $\Delta((i_p,j_p),n-1)$ contains one of the points $(u_1,u_2-1), (u_1,v_2-1), (v_1,u_2-1), (v_1,v_2-1)$. We treat the case $(u_1,u_2-1)\in \Delta((i_p,j_p),n-1)$, the remaining cases are similar.

We have a chain
\[
0\le u_1-i_p \le u_2-1-j_p\le n-1,
\]
which clearly implies
\[
0\le u_1-i_p \le u_2-j_p\le n,
\]
namely $(u_1,u_2)\in \Delta((i_p,j_p),n)$. This contradicts \eqref{eq_noedges}.

Hence $u_1v_1,(u_2-1)(v_2-1)$ form a $2K_2$ in $G_{n+r-1}$ and $\indmatch(G_{n+r-1})=2$, as desired.

From Step 1 and Step 2, for any $n\ge 2r$, the equality $\indmatch(G_{n+r})=2$ holds if and only if $\indmatch(G_{n+r+1})=2$. Since $I_n$ is a nonzero edge ideal, $1\le \indmatch(G_{n+r}) \le 2$, for all $n\ge 2r$, and we deduce that  $\indmatch(G_{n+r})=\indmatch(G_{n+r+1})$ for all such $n$. The proof is concluded.
\end{proof}

\section{Induced anticycles}
\label{sect_inducedanticyc}

Our main goal in this section is to show that $G_n$ does not contain short induced anticycles of length at least $5$ for all sufficiently large $n$.
\begin{prop}
\label{prop_noCmc}
Let $\Icc=(I_n)_{n\ge 1}$ be an $\Inc$-invariant chain of eventually nonzero edge ideals with the stability index $\ind \Icc=r\ge 1$. Let $G_n$ be the graph corresponding to $I_n$. Let $m\ge 5$ be an integer. Then for all $n\ge mr$, the complement graph $G_n^c$ of $G_n$ has no induced cycle $C_m$.

In particular, for all $n\ge 5r$, the complement graph $G_n^c$ of $G_n$ has no induced cycle $C_m$ for any integer $m$ such that $5\le m \le \dfrac{n}{r}$.
\end{prop}
The main ingredients to the proof are Lemma \ref{lem_twopoints_corners}  and the following lemma. 
\begin{lem}
\label{lem_inducedanticyc}
With usual notations, let $n\ge 2r$ be an integer. Let $H$ be an induced subgraph of $G_{n+r}$ such that $H$ contains an induced anticycle $C_m^c$ \textup{(}where $m\ge 5$\textup{)} whose consecutive vertices are labeled as $a_1,\ldots,a_m \in \{1,\ldots,n+r\}$ such that $a_1=\min\{a_1,\ldots,a_m\}$. Then $\max\{a_1,a_2,\ldots,a_m\}=\max\{a_2,a_m\}$.
\end{lem}
\begin{proof}
Again, denote $E(G_r)=\{\{i_1,j_1\},\ldots,\{i_s,j_s\}\}$ where $1\le i_p < j_p\le r$ for $p=1,\ldots,s$. Denote $\Delta_p:=\Delta_{p,n}:=\Delta((i_p,j_p),n)$ for $1\le p\le s$.

It suffices to show that if $a_2\le a_m$, then $a_m=\max\{a_1,a_2,\ldots,a_m\}$. Indeed, for the case $a_m < a_2$, reindex the vertices as follows: $b_1=a_1, b_i=a_{m+2-i}$ for $2\le i\le m$. Then $b_1=\min \{b_1,b_2,\ldots,b_m\}$ and $b_2= a_m < a_2=b_m$. Thus we may assume that $a_2\le a_m$, and we show that $a_m=\max\{a_1,a_2,\ldots,a_m\}$.

By the definition of the stability index, for each pair of integers $(p,q)$ such that $1 \le p, q\le m, q-p\ge 2$ and $(p,q)\neq (1,m)$, there exists $(i_{pq},j_{pq})\in G_r$ where $1\le i_{pq} <j_{pq}\le r$ such that $(a_p,a_q)^\le \in \Delta((i_{pq},j_{pq}),n)$.

For simplicity, denote $\Delta_{pq}=\Delta((i_{pq},j_{pq}),n)$. We let $a_{m+1}=a_1$, $A_{pq}=(i_{pq},j_{pq})$ for $1\le p<q\le m, q\ge p+2, (p,q)\neq (1,m)$.

\textbf{Claim 1:}  For all $1\le p\le m$, none of the points $(a_p,a_{p+1}), (a_{p+1},a_p)$ belongs to 
$$
\Gamma=\mathop{\bigcup_{1\le p<q\le r}}_{q\ge p+2, (p,q)\neq (1,m)} \Delta_{pq}.
$$

\emph{Proof of Claim 1:} First assume that for some $1\le t\le m$, we have 
\[
(a_p,a_{p+1})\in\Gamma \subseteq \bigcup_{t=1}^s \Delta_t.
\]
Therefore, there exists an integer $t\in\{1,\ldots,s\}$ such that $(a_p,a_{p+1})\in \Delta_t.$ Since $i_t<j_t$, $a_p<a_{p+1}$. 
Now
\[
\{a_p,a_{p+1}\} \in E(G_{n+r}),
\]
and $H$ is an induced subgraph of $G_{n+r}$, thus $\{a_p,a_{p+1}\}\in E(H)$, a contradiction. Hence $(a_p,a_{p+1})\notin \Gamma$. Similarly for $(a_{p+1},a_p)$, we finish the proof of Claim 1.

Now consider the two points $(a_1,a_3)=(a_1,a_3)^{\le} \in \Delta_{13}, (a_2,a_m)=(a_2,a_m)^{\le} \in \Delta_{2m}$. 

Observe that $i_{13}< i_{2m}$. Indeed, assume the contrary $i_{2m} \le i_{13}$. By Claim 1, $(a_2,a_3)$, $(a_1,a_m)\notin \Delta_{2m} \cup \Delta_{13} $. Note that $n\ge 2r \ge 2\max\{j_{2m}, j_{13}\}$. Hence using \Cref{lem_twopoints_corners} with $(A_{2m}, A_{13}, (a_2,a_m), (a_1,a_3))$ in place of $(A_1,A_2, (u_1,v_1), (u_2,v_2))$, we get 
\begin{align*}
a_2 < i_{13}\le  a_1,
\end{align*}
contradicting the minimality of $a_1$.

\begin{figure}[ht!]
 \includegraphics[width=35ex]{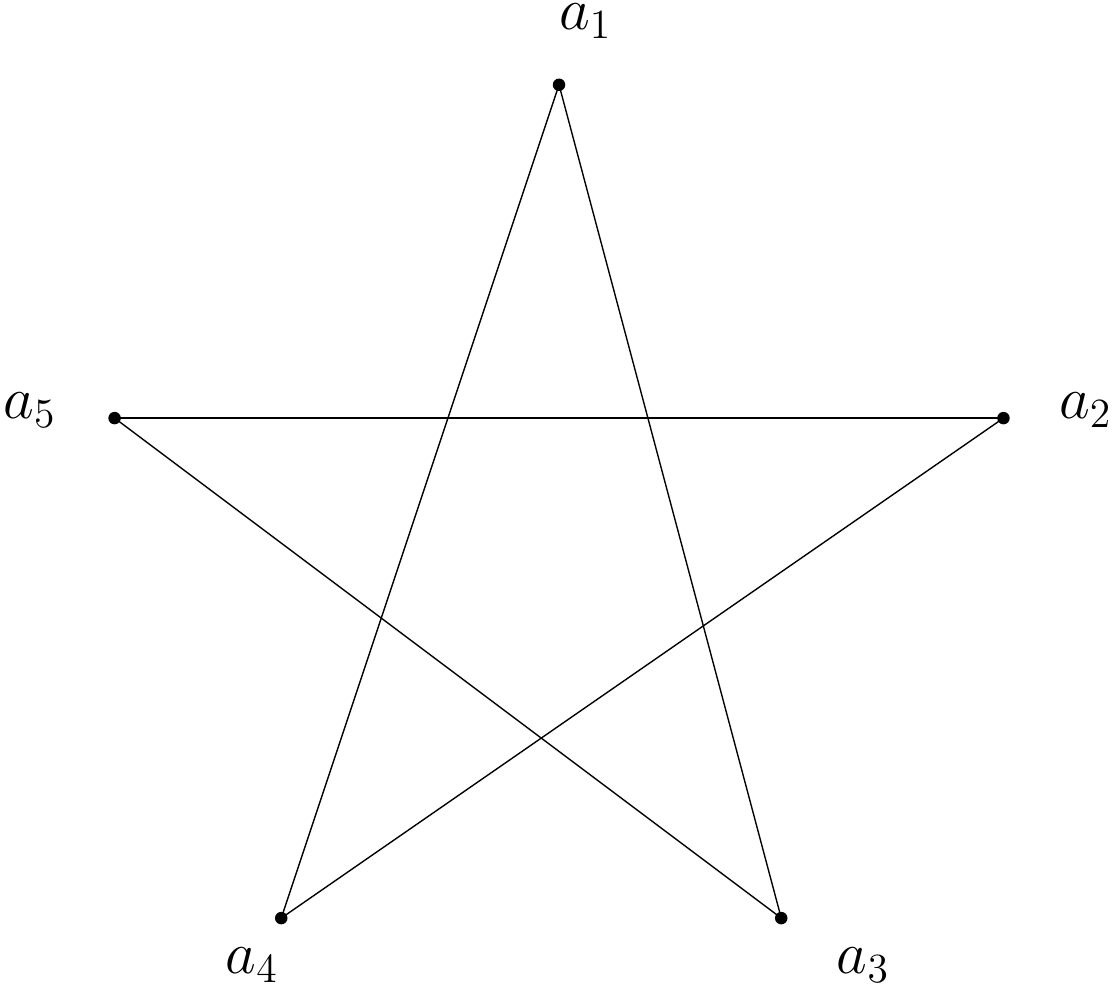}
 \caption{An induced anticycle $C_5^c$}
	\label{Anticycle $C_5$} 
\end{figure}

This implies that $i_{13}<i_{2m}$. The last inequality puts us in a position to apply \Cref{lem_twopoints_corners} for the data $(A_{13}, A_{2m}, (a_1,a_3), (a_2,a_m))$, as by Claim 1, it holds that $(a_1,a_m), (a_2,a_3) \notin \Delta_{13} \cup \Delta_{2m}$. The results are the following inequalities
\[
a_3  \le j_{13}+n < a_m.
\]
In particular, $a_3<a_m$. Continuing this argument inductively by looking at $(a_1,a_{i+1})$ and $(a_i,a_m)$ for $3\le i\le m-2$, we get $a_m=\max\{a_p: 1\le p\le m\}$, as desired.
\end{proof}
We are ready now for the 
\begin{proof}[Proof of \Cref{prop_noCmc}]
It suffices to prove the first assertion.

Assume the contrary, that there exists $n\ge (m-1)r$ such that $G_{n+r}^c$ contained an induced cycle $a_1,a_2,\ldots,a_m,a_1$ where $1\le a_i\le n+r, a_i\neq a_j$ for $i\neq j$. We may assume that $a_1=\min\{a_1,a_2,\ldots,a_m\}$.

Hence $(a_p,a_q)\in G_{n+r}$ for all $1\le p<q\le m$ such that $q-p\ge 2$ and $(p,q)\neq (1,m)$. By the definition of the stability index, for each pair of integers $(p,q)$ such that $1 \le p, q\le m, q-p\ge 2$ and $(p,q)\neq (1,m)$, there exists $(i_{pq},j_{pq})\in G_r$ where $1\le i_{pq} <j_{pq}\le r$ such that $(a_p,a_q)^\le \in \Delta((i_{pq},j_{pq}),n)$. As in the proof of \Cref{lem_inducedanticyc}, we have

\textbf{Claim 1:}  For all $1\le p\le m$, none of the points $(a_p,a_{p+1}), (a_{p+1},a_p)$ belongs to 
$$
\Gamma=\mathop{\bigcup_{1\le p<q\le r}}_{q\ge p+2, (p,q)\neq (1,m)} \Delta_{pq}.
$$

For $2\le p\le m-2$, note that $(a_p,a_{p+1})\notin \Delta_{1(p+1)}$ by Claim 1. Hence the following chain is false
\[
0\le a_p-i_{1(p+1)} \le a_{p+1}-j_{1(p+1)} \le n.
\]
Since $a_p\ge a_1$ and $(a_1,a_{p+1})\in \Delta_{1(p+1)}$, the first inequality is true. The last inequality is also true, so the middle one is false. Equivalently, we get
\begin{equation}
 \label{eq_ineq_ap+1small}
 a_{p+1}-a_p< j_{1(p+1)} -i_{1(p+1)} \quad \text{for $2\le p\le m-2$}.
\end{equation}
For $3\le p\le m-1$, note that $(a_{p+1},a_p) \notin \Delta_{1p}$ by Claim 1. Therefore the following chain is invalid
\[
0\le a_{p+1}-i_{1p} \le a_p-j_{1p}\le n. 
\]
As $a_{p+1}\ge a_1$, the first and last inequalities are valid, hence the middle one is not. Thus
\begin{equation}
\label{eq_ineq_apsmall}
 a_p - a_{p+1} < j_{1p}-i_{1p} \quad \text{for $3\le p\le m-1$}.
\end{equation}
Applying \Cref{lem_twopoints_corners} with the data $(A_{13}, A_{2m}, (a_1,a_3), (a_2,a_m))$ as in the proof of \Cref{lem_inducedanticyc}, we get
\begin{align}
a_1< & i_{2m} \le a_2, \label{eq_ineq_a1_small}\\
a_3 & \le j_{13}+n < a_m. \label{eq_ineq_a3m}
\end{align}

\textsf{Case I:} $a_2\le a_m$. By \Cref{lem_inducedanticyc}, $a_m=\max\{a_1,\ldots,a_m\}$.

\textbf{Case I.1:} $a_2 \le a_{m-1}$. By Claim 1, $(a_1,a_2)\notin \Delta_{1(m-1)}$ hence the following chain is false
\[
0\le a_1-i_{1(m-1)} \le a_2-j_{1(m-1)} \le n.
\]
The first inequality in the chain holds since $(a_1,a_{m-1})\in \Delta_{1(m-1)}$, and the last inequality also holds since $a_2\le a_{m-1}\le n+j_{1(m-1)}$. Hence the middle inequality is false, namely
\begin{equation}
 \label{eq_ineq_a2small}
 a_2-a_1 < j_{1(m-1)}-i_{1(m-1)}.
\end{equation}
Now $(a_{m-1},a_m)\notin \Delta_{2m}$ by Claim 1 so the following chain is false
\[
0\le a_{m-1}-i_{2m} \le a_m-j_{2m} \le n.
\]
Since $(a_2,a_m)\in \Delta_{2m}$ and $a_2 \le a_{m-1}$ by the hypothesis of Case I.1, we see that the first and last inequalities of the last chain are true. Thus the middle one is false, namely
\begin{equation}
 \label{eq_ineq_amsmall}
 a_m-a_{m-1} < j_{2m}-i_{2m}.
\end{equation}
Using \eqref{eq_ineq_a2small}, \eqref{eq_ineq_ap+1small} multiple times, and \eqref{eq_ineq_amsmall}, we get
\begin{align*}
a_m-a_1 &< (j_{1(m-1)}-i_{1(m-1)})+\sum_{p=2}^{m-2} \left( j_{1(p+1)} -i_{1(p+1)}\right)+(j_{2m}-i_{2m})\\
        & < (m-2)\max\{j_{pq}: 1\le p<q\le m, q\ge p+2, (p,q)\neq (1,m)\} \\
        & \le (m-2)r.
\end{align*}
Together with \eqref{eq_ineq_a1_small} and \eqref{eq_ineq_a3m}, we get
\[
n+j_{13}-i_{2m}< a_m-a_1 < (m-2)r,
\]
so $n< (m-2)r+i_{2m} \le (m-1)r$. This contradicts the hypothesis $n\ge (m-1)r$ from above.

\textbf{Case I.2:} $a_2 > a_{m-1}$.

Consider the points $(a_{m-1},a_2)\in \Delta_{2(m-1)}$  and $(a_3,a_m)\in \Delta_{3m}$. Since none of the points $(a_3,a_2)$, $(a_{m-1},a_m)$ belongs to  $\Gamma$, by \Cref{lem_twopoints_corners} and the fact that $a_m=\max\{a_i: 1\le i\le m\}$, we get $i_{2(m-1)}<i_{3m}$, and 
\begin{equation}
 \label{eq_ineq_am-1small}
 a_{m-1} < i_{3m} \le a_3.
\end{equation}

Consider the points $(a_1,a_{m-2})\in \Delta_{1(m-2)}$ and $(a_{m-1},a_2)\in \Delta_{2(m-1)}$. Arguing as above using \Cref{lem_twopoints_corners} and the fact that $a_1\le a_{m-1}$, we get
\begin{equation}
\label{eq_ineq_a2large}
 a_2 > n+j_{1(m-2)}\ge a_{m-2}.
\end{equation}

We have $(a_3,a_2)\notin \Delta_{2(m-1)}$ per Claim 1, so the following chain is invalid
\[
0\le a_3-i_{2(m-1)} \le a_2-j_{2(m-1)}\le n. 
\]
Thanks to \eqref{eq_ineq_am-1small}, $a_3>a_{m-1}\ge i_{2(m-1)}$ and $a_2-j_{2(m-1)}\le n$, the middle inequality is false, namely
\begin{equation}
\label{eq_ineq_a2-a3small}
 a_2- a_3 < j_{2(m-1)}-i_{2(m-1)}.
\end{equation}
Combining \eqref{eq_ineq_am-1small}, \eqref{eq_ineq_a2large}, \eqref{eq_ineq_apsmall}, and \eqref{eq_ineq_a2-a3small}, we obtain
\begin{align*}
n+j_{1(m-2)}-i_{3m} &< a_2-a_{m-1} \\
                    &=(a_2-a_3)+\sum_{p=3}^{m-2}(a_p-a_{p+1})\\
                    &< j_{2(m-1)}-i_{2(m-1)}+\sum_{p=3}^{m-2}(j_{1p}-i_{1p} ) \quad \text{(by \eqref{eq_ineq_a2-a3small} and \eqref{eq_ineq_apsmall})}\\
                    & < (m-3) \max\{j_{pq}: 1\le p<q\le m, q\ge p+2, (p,q)\neq (1,m)\} \\
        & \le (m-3)r.
\end{align*}
Hence $n < (m-3)r+i_{3m} \le (m-2)r,$ a contradiction. This finishes Case I.

\textsf{Case II:} $a_2 > a_m$. Relabel the vertices as follows: $b_1=a_1, b_i=a_{m+2-i}$ for $2\le i\le m$. Then $b_1=\min\{b_1,\ldots,b_m\}$ and $b_2 < b_m$. Arguing as for Case I, we finish Case II and the proof.
\end{proof}
Recall that a graph $G$ is \emph{weakly chordal} if neither $G$ nor $G^c$ contains an induced cycle of length at least $5$. The following example shows that one cannot improve the conclusion of \Cref{prop_noCmc} in general.
\begin{ex}
\label{ex_not_weaklychordal}
In general, it may happen that the graph $G_n$ corresponding to $I_n$ is not weakly chordal for all large $n$.

Take $I_4=(x_1x_3,x_2x_4)$ and $I_n=\Inc_{4,n} (I_4)$ for all $n\ge 4$. Explicitly, for all $n\ge 4$, 
\[
I_n=\left(x_ix_j: 1\le i< j\le n, j-i\ge 2, (i,j)\neq (1,n) \right).
\]
For each $n\ge 4$, $G_n^c$ contains the induced cycle with vertices $1, 2, 3,\ldots,n$, of length $n$, as one may easily check.

In particular, for all $n\ge 5$, $G_n$ is not weakly chordal. Moreover for each $n\ge 6$, 
$\reg I_n \ge \reg I(C_n^c)=3.$ Together with \Cref{thm_reg<=3} below, we conclude that $\reg I_n=3$ for all $n\ge 6$.

On the other hand, using \Cref{lem_twopoints_induced2K2} for $(i_1,j_1)=(1,3), (i_2,j_2)=(2,4)$,  we deduce that $G_{n+4}$ has no induced $2K_2$ for all 
$n\ge \max\{2(3-1+2-4),3+4-2\times 1,1\}=5.$ In other words, $\indmatch(G_n)=1$ for all $n\ge 9$. Therefore we get $\reg(I_n) =3 > 2= \indmatch(G_n)+1$ for all $n\ge 9$.
\end{ex}

\section{Long induced anticycles}
\label{sect_longanticycles}
In this section, we fix the following notations.
\begin{notn}
\label{notn_chainofedgeids}
Let $\Icc=(I_n)_{n\ge 1}$ be an $\Inc$-invariant chain of eventually nonzero edge ideals. Let $r=\ind(\Icc)$ be the stability index of the chain. Denote $E(G_r)=\{\{i_1,j_1\},\ldots,\{i_s,j_s\}\}$ where $1\le i_t<j_t \le r,\ i_1\le i_2\le \cdots\le i_s$ and if $i_t=i_{t+1}$ then $j_t<j_{t+1}$. Denote $q=\max\{1\le t\le s: i_t=i_1\}$.
\end{notn}
We have seen in \Cref{prop_noCmc} that $G_n$ does \emph{not} contain \emph{short} induced anticycles, namely anticycle of the form $C_m^c$, for $5\le m \le \frac{n}{r}$. The following result provides the complementary information, revealing that $G_n$ does contain \emph{long} induced anticycles under some extra condition on the edges of $G_r$. This result is important later in \Cref{sect_converge} in the characterization of the case that $\reg I_n=2$ for all $n\gg 0$.
\begin{prop}
\label{prop_reg3_specialcase1}
Keep using \Cref{notn_chainofedgeids}.  Assume $\max\{j_1,\ldots,j_s\}=j_q+1$ and $\min\{j_t-i_t: t=1,\ldots,s\}\ge 2$.  Then for all  $n\ge 3r$, there is an inequality $\reg(I_n)\ge 3$.
\end{prop}
The significance of the two hypotheses $\max\{j_1,\ldots,j_s\}=j_q+1$ and $\min\{j_t-i_t: t=1,\ldots,s\}\ge 2$ will be apparent in the equivalence between statements (1) and (3) of \Cref{thm_convergence}.
 
The most difficult part of the proof is certain construction of long induced anticycles in the graph $G_n$, specifically \Cref{constr_inducedanticycle}. The following notations are used throughout the proof and the construction.
\begin{notn}
\label{notn_setsandindices}
For simplicity, denote $\Delta_t=\Delta((i_t,j_t),n)$ for any $t=1,\ldots,s$.  Set $J=\{1,2,\ldots,s\},\ J_1=\{t\in J\mid j_t-i_t=\min\{j_1-i_1,\ldots,j_s-i_s\}\}$. Denote
\begin{align*}
b=\min\{t\mid j_t=\max\{j_1,\ldots,j_s\}\}, & \quad B=\max\{t\mid j_t=\max\{j_1,\ldots,j_s\}\},\\
h=\min J_1, & \quad H=\max J_1.
\end{align*}
\end{notn}
The following observation is useful to clarify the sets and indices in \Cref{notn_setsandindices}.
\begin{lem}
\label{lem_BandJ1}
With \Cref{notn_setsandindices}, we have $j_H\le j_B$. Moreover, the following statements are equivalent:
\begin{enumerate}[\quad \rm (1)]
 \item $B=H$;
 \item $B\in J_1$;
 \item $j_H=j_B$.
\end{enumerate}
\end{lem}
\begin{proof}
Clearly $j_H\le j_B=\max\{j_1,\ldots,j_s\}$. Since $H=\max J_1$, it is also clear that (1) $\Longrightarrow$ (2). We prove that (2) $\Longrightarrow$ (3) and (3) $\Longrightarrow$  (1).

(2) $\Longrightarrow$ (3): Assume that $B\in J_1$. Then $B\le H=\max J_1$ and $j_B-i_B=j_H-i_H$. From \Cref{notn_chainofedgeids}, $j_H-j_B=i_H-i_B\ge 0$ as $H\ge B$. But the definition of $B$ yields $j_H\le j_B$, hence $j_H=j_B$. Thus (3) is true.

(3) $\Longrightarrow$  (1): Assume that $j_H=j_B$. The definition of $B$ implies $H\le B$, so $i_H\le i_B$. In particular, $j_H-i_H\ge j_B-i_B$. But $j_H-i_H=\min\{j_1-i_1,\ldots,j_s-i_s\}$, so $j_H-i_H=j_B-i_B$, and hence $i_H=i_B$. The only situation when $i_H=i_B$ and $j_H=j_B$ are both valid is $B=H$. Hence (1) is true. The proof is concluded.
\end{proof}

The key step in the proof of \Cref{prop_reg3_specialcase1} is accomplished by the following constructions (\ref{constr_Jsets}, \ref{constr_Ksets} and \ref{constr_inducedanticycle}). The last constructions employ the first two, and we clarify all of them in \Cref{ex_anticycle}.

\begin{constr}
\label{constr_Jsets}
Keep using \Cref{notn_setsandindices}. In this construction, \emph{we assume that $i_b\le i_h$}. 

Note that in this case $i_1=i_q < i_b$, since $j_q=\max\{j_t: i_t=i_1\} < \max\{j_t: t=1,\ldots,s\} = j_b$. Hence $q\in J\setminus J_1$. Indeed, if not, then  $j_q-i_q=\min\{j_t-i_t: 1\le t\le s\}$. This implies $q=1$, and  
$
h=\min \{t: j_t-i_t=\min\{j_1-i_1,\ldots,j_s-i_s\}\}=1,
$
and by the assumption, $i_b \le i_h=i_1$, a contradiction. Thus $q\in J\setminus J_1$ and $i_q < i_b\le i_h$. This justifies the first step of \Cref{alg_Jsets}, which is useful for sorting  the triangles $\Delta_p$.
\begin{algorithm}[H]
\caption{Rearrange the positions of the triangles $\Delta_p$ with $i_b\le i_p\le i_h$}
\label{alg_Jsets}
\begin{flushleft}
\textbf{Input}: The set of edges $\{\{i_1,j_1\},\ldots,\{i_s,j_s\}\}$ of $G_r$ given that $i_b\le i_h$. \\
\textbf{Output}: A sequence of sets $J_1,\ldots,J_\beta$ and a sequence of integers $u_1,\ldots,u_\beta \in \{1,\ldots,s\}$.
\end{flushleft}
\begin{algorithmic}
\State{$t:=1$  }	
\State{$u_1:=h (=\min J_1)$  }	
\While{$J_t\neq \emptyset$ and $i_{u_t}\ge i_{b}$}
\State {$J_{t+1}:=\{p \mid j_p-i_p=\min\{j_a-i_a\mid a\in J\setminus \cup_{i=1}^{t} J_i\;\text{and}\; i_a<i_{u_{t}}\}\}$}
\State{$u_{t+1}:=\min  J_{t+1} $}
\State{$t:=t+1$}
\EndWhile
\end{algorithmic}
\end{algorithm}

Apply Algorithm~\ref{alg_Jsets} and assume that $(\beta-1)$ is the number of repeating steps of this algorithm. By what said above, $\beta\ge 2$, as $J_2\neq \emptyset$. The algorithm gives
\[
J_{\beta}=\{p \mid j_p-i_p=\min\{j_a-i_a\mid a\in J\setminus \cup_{t=1}^{\beta-1} J_t\;\text{and}\; i_a<i_{u_{\beta-1}}\}\}, 
\]
and $u_{\beta}= \min J_{\beta}$. Note that $q\in J\setminus \bigcup_{t=1}^{\beta-1} J_t$ since $i_q< i_b$ and $i_t\ge i_b$ for any $t\in J_1 \cup \cdots \cup J_{\beta-1}$. Hence $J_{\beta}\neq \emptyset$, so because the algorithm stops, we get $i_{u_{\beta}} <i_b$.

It follows that  we obtain a sequence $h=u_1,u_2,\ldots,u_{\beta}\in J$ such that 
\begin{align*}
i_{u_{\beta}} &<i_b\le i_{u_{\beta-1}}<\cdots<i_{u_2}<i_{u_1}=i_h,\\
j_{u_{t+1}}-i_{u_{t+1}}&>j_{u_t}-i_{u_t}\ge 2\quad \text{for}\; t=1,\ldots,\beta-1.
\end{align*}
\end{constr}

\begin{constr}
\label{constr_Ksets}
Keep using \Cref{notn_setsandindices}.  The following algorithm is used to rearrange the positions of the triangles $\Delta_p$. 

\begin{algorithm}[H]
	\caption{Rearrange the positions of the triangles $\Delta_p$ with  $j_{H}<j_p\le j_B$}
	\label{alg_Ksets}
    \begin{flushleft}
    \textbf{Input}: The set of edges $\{\{i_1,j_1\},\ldots,\{i_s,j_s\}\}$ of $G_r$. \\
\textbf{Output}: A sequence of sets $J_1=K_1, K_2, \ldots,K_\gamma$ and a sequence of integers $v_1,\ldots,v_\gamma \in \{1,\ldots,s\}$.
  \end{flushleft}
	\begin{algorithmic}
		\State{$K_1:=J_1$  }
		\State{$t:=1$  }	
		\State{$v_1:=H (=\max K_1)$  }
		\While{$K_t\neq \emptyset$ and $j_{v_t} < j_{B}$}
		\State {$K_{t+1}:=\{p \mid j_p-i_p=\min\{j_a-i_a\mid a\in J\setminus \cup_{i=1}^{t} K_i\;\text{and}\; j_a>j_{v_{t}}\}\}$}
		\State{$v_{t+1}:=\max K_{t+1} $}
		\State{$t:=t+1$}
		\EndWhile
	\end{algorithmic}
\end{algorithm}
In the case $j_H=j_B$, note that $H=B$ by \Cref{lem_BandJ1}, and therefore \Cref{alg_Ksets} does nothing. In other words, $K_1=J_1$, $\gamma=1$, and $v_1=H=B$ in this case.

Consider the case $j_H<j_B$. Apply Algorithm~\ref{alg_Ksets} and assume that $(\gamma-1)$ is the number of repeating steps of this algorithm. Thanks to \Cref{lem_BandJ1}, $B\notin K_1=J_1$ and $j_{v_1}=j_H < j_B$. So $K_2\neq \emptyset$ and $\gamma \ge 2$. 

\textbf{Claim 1:} If $B\in K_t$ for some $2\le t\le \gamma$, then $B=\max K_t=v_t$.

Indeed, since $B\in K_t$, $B\le \max K_t=v_t$. From the setting of \Cref{notn_chainofedgeids}, $i_B\le i_{v_t}$. Since $B\in K_t$, $j_B-i_B=j_{v_t}-i_{v_t}$, so $j_{v_t}-j_B=i_{v_t}-i_B\ge 0$. The definition of $B$ implies $j_{v_t}-j_B\le 0$, so $0=j_{v_t}-j_B=i_{v_t}-i_B$. This forces $B=v_t$.

\textbf{Claim 2:} It holds that $B\in J\setminus \bigcup_{t=1}^{\gamma-1} K_t$. In particular $B\in K_\gamma$ and
\begin{equation}
\label{eq_vgammaB}
v_{\gamma}=B.
\end{equation}

Indeed, we have since that $B\notin K_1$. Assume that $B\in K_t$ for some $2\le t\le \gamma-1$, then $B=v_t$ by Claim 1. But then $j_{v_t}=j_B$, contradicting the description of \Cref{alg_Ksets}. Hence $B\in J\setminus \bigcup_{t=1}^{\gamma-1} K_t$. The description of the algorithm yields $j_{v_{\gamma-1}}<j_B$, so $K_\gamma \neq \emptyset$.  Since the algorithm stops, we deduce $j_{v_\gamma}=j_B$. The definition of $B$ in \Cref{notn_setsandindices} yields $v_\gamma \le B$, so $i_{v_\gamma}\le i_B$. In particular, $j_B-i_B \le j_{v_\gamma}-i_{v_\gamma}$, so the definition of $K_\gamma$, that reads
\[
K_\gamma=\{p \mid j_p-i_p=\min\{j_a-i_a\mid a\in J\setminus \cup_{i=1}^{\gamma-1} K_i\;\text{and}\; j_a>j_{v_{\gamma-1}}\}\}
\]
implies that $B\in K_\gamma$. Using Claim 1, we conclude that $B=\max K_\gamma=v_\gamma$.

So at the end of Algorithm~\ref{alg_Ksets}, we obtain a sequence of indices $H=v_1,v_2,\ldots,v_{\gamma}=B$ such that
\begin{align*}
j_B = j_{v_{\gamma}} &>\cdots>j_{v_2}>j_{v_1}=j_H,\\
j_{v_{t+1}}-i_{v_{t+1}}& >j_{v_t}-i_{v_t}\ge 2\quad \text{for all}\; t=1,\ldots,\gamma-1.
\end{align*}
\end{constr}

\begin{constr}
\label{constr_inducedanticycle}
Keep using \Cref{notn_setsandindices}. Let $h=u_1,u_2,\ldots,u_{\beta}$ be as in \Cref{constr_Jsets} and $H=v_1,v_2,\ldots,v_{\gamma}=B$ be as in \Cref{constr_Ksets} (recalling Equation \eqref{eq_vgammaB}). Now for $n\ge 2r$, we construct an induced anticycle $C_m^c$ in $G_{n+r}$ by considering the following two cases.

\noindent\textbf{Case I: $i_b \le i_h.$}
We employ Algorithm~\ref{alg:anticycle_initialvertices} to construct the initial vertices of $C_m^c$ in $G_{n+r}.$ Note that the numbers $\beta, h=u_1, u_2,\ldots,u_{\beta}$ in \Cref{alg:anticycle_initialvertices} come from the output of \Cref{alg_Jsets}.

\begin{algorithm}[H]
	\caption{Generate the initial vertices of an induced anticycle in $G_{n+r}$}
	\label{alg:anticycle_initialvertices} 
\begin{flushleft}
\textbf{Input}: The set of edges $\{\{i_1,j_1\},\ldots,\{i_s,j_s\}\}$ of $G_r$ given that $\min\{j_t-i_t \mid t=1,\ldots,s\}\ge 2$ and $i_b\le i_h$, and an integer $n\ge 2r$. \\
\textbf{Output}: A sequence of numbers $a_1,\ldots,a_d,a_{d+1} \in \{1,2,\ldots,n+r\}$.
\end{flushleft}
	\begin{algorithmic} 
		\State{$\epsilon:= \max\{t \in \Z_{\ge 0} \mid t(j_{u_{\beta}} - i_{u_{\beta}}-1) + i_{u_{\beta}}< i_{b}\}$}
		\State{$a_1:=\epsilon(j_{u_{\beta}} - i_{u_{\beta}}-1) + i_{u_{\beta}}$}
		\State{${\rm term}:=a_1$}
		\State{$d:=1$}
		\While{${\rm term} < i_h$}
		\State{$t=\min\{\alpha\in\{1,\ldots,\beta\} \mid i_{u_\alpha}\le {\rm term}\}$}
		\State{${\rm term}:={\rm term}+ j_{u_t} -i_{u_t} -1$}
		\State{$a_{d+1}:={\rm term}$}
		\State{$d:=d+1$}
		\EndWhile
	\end{algorithmic}
\end{algorithm}
Note that
\begin{equation}
\label{eq_a2-a1}
a_2=a_1+j_{u_\beta} -i_{u_\beta} -1.
\end{equation}
Indeed, by \Cref{constr_Jsets}, $i_{u_\beta} < i_b \le i_{u_{\beta-1}}$, so $a_1=\epsilon(j_{u_{\beta}} - i_{u_{\beta}}-1) + i_{u_{\beta}} <i_b \le i_{u_{\beta-1}}$. Hence at the step $d=1$, 
\[
t=\min\{\alpha\in\{1,\ldots,\beta\} \mid i_{u_\alpha}\le a_1\}=\beta.
\]
From this, \eqref{eq_a2-a1} follows.

At the end of  Algorithm~\ref{alg:anticycle_initialvertices}, we get a sequence $a_1,\ldots,a_{d+1}$ satisfying 
\[
i_q\le i_{u_{\beta}}\leq a_1<i_b\le a_2<a_3<\cdots<a_d< i_h\le a_{d+1}=a_d+j_{u_t}-i_{u_t}-1,
\] 
where $t=\min\{\alpha\in\{1,\ldots,\beta\} \mid i_{u_\alpha}\le a_d\}$.

At this point, note that $v_\gamma=B$ per \eqref{eq_vgammaB}, so
\[
a_{d+1}=a_d+j_{u_t}-i_{u_t}-1 < i_h+j_{u_t}-i_{u_t}-1 < 2r < n+i_{v_\gamma} = n +i_B,
\]
using the hypothesis on $n$. Thus we may apply Algorithm~\ref{alg:anticycle_finalvertices} with $a_\textsf{index}:=a_{d+1}$. Note that the numbers $\gamma, H=v_1, v_2,\ldots,v_{\gamma}=B$ in \Cref{alg:anticycle_finalvertices} come from the output of \Cref{alg_Ksets}.

\begin{algorithm}[H]
	\caption{Generate the final vertices of an induced anticycle in $G_{n+r}$}
	\label{alg:anticycle_finalvertices}
	\begin{flushleft}
\textbf{Input}: The set of edges $\{\{i_1,j_1\},\ldots,\{i_s,j_s\}\}$ of $G_r$ and integers $n\ge 2r$ and $i_h \le a_\textsf{index} \le n+r$.  \\
\textbf{Output}: A sequence of numbers $a_\textsf{index},a_{\textsf{index}+1}\ldots,a_m \in \{1,2,\ldots,n+r\}$.
	\end{flushleft}
	\begin{algorithmic}
		\State{${\rm term}:= a_{\textsf{index}}$}
		\State{$\nu:=1$}
		\While{${\rm term}\le n+i_B$}
		\State{$t=\min\{\alpha\in\{1,\ldots,\gamma\} \mid {\rm term} \le n+i_{v_\alpha}\}$}
		\State{${\rm term}:= {\rm term}+j_{v_{t}} - i_{v_{t}}-1$}
		\State{$a_{\textsf{index}+\nu}:={\rm term}$}
		\State{$\nu:=\nu+1$}
		\EndWhile
		\State{$m:=\nu+1$}
		\State{$a_{m}:=n+j_B$}
	\end{algorithmic}
\end{algorithm}
\end{constr}

At the end of Algorithm \ref{alg:anticycle_finalvertices}, we get a sequence $a_{d+2},\ldots,a_m$ satisfying
\[
a_{d+1}<a_{d+2}<\cdots <a_{m-2}\le n+i_B< a_{m-1} <a_m=n+j_B.
\]

\noindent\textbf{Case II: $i_{h}<i_{b}.$}
Set $\epsilon=\max\{t \in \Z_{\ge 0} \mid t(j_{h}-i_{h}-1)+i_{h}<i_b   \}$ and define 
\begin{align*}
a_1&=\epsilon (j_{h}-i_{h}-1)+i_{h}<i_b,\\
a_2&= a_1+j_{h}-i_{h}-1\ge i_b.
\end{align*}
We are in place to apply Algorithm~\ref{alg:anticycle_finalvertices} with $a_\textsf{index}:=a_2$  since 
$$
a_2=a_1+j_h-i_h-1<i_b+j_h-i_h-1<2r < n+i_B.
$$
From the last algorithm, we obtain a sequence $a_1,a_2,\ldots,a_m$ satisfying
\[
i_h\le a_1< i_b\le a_2<a_3<\cdots <a_{m-2}\le n+i_B< a_{m-1} < a_m=n+j_B.
\]

With Lemmas \ref{lem_length_inducedanticycle}, \ref{lem_consecutivevertices}, \ref{lem_non_consecutivevert} below, we show that $a_1,\ldots,a_m$ in \Cref{constr_inducedanticycle} form an induced anticycle in $G_{n+r}$ of length $\ge 4$ for all $n\ge 2r$. But first, we consider an example of the Constructions \ref{constr_Jsets} -- \ref{constr_inducedanticycle}.

\begin{ex}
\label{ex_anticycle}
Consider the $\Inc$-invariant chain $\Icc=(I_n)_{n\ge 1}$ of edge ideals with stability index $9$ and 
$$
E(G_9)=\{\{1,5\},\{1,8\},\{2,9\},\{3,6\},\{4,7\},\{5,9\}\}.
$$
Clearly the conditions $j_q+1=\max\{j_1,\ldots,j_s\}$ and $\min \{j_t-i_t \mid t=1,\ldots,s\} \ge 2$ of \Cref{prop_reg3_specialcase1} are satisfied. In this case $i_1=1, q=2$, $j_2=8$, $\max\{j_1,\ldots,j_6 \}=9=8+1$, and $\min \{j_t-i_t \mid t=1,\ldots,6\}=3$.

The plot of the corresponding points $(i_t,j_t), t=1,\ldots,6$ is as in \Cref{fig_Anticycle_ex}.
\begin{figure}[ht!]
 \includegraphics[width=66ex]{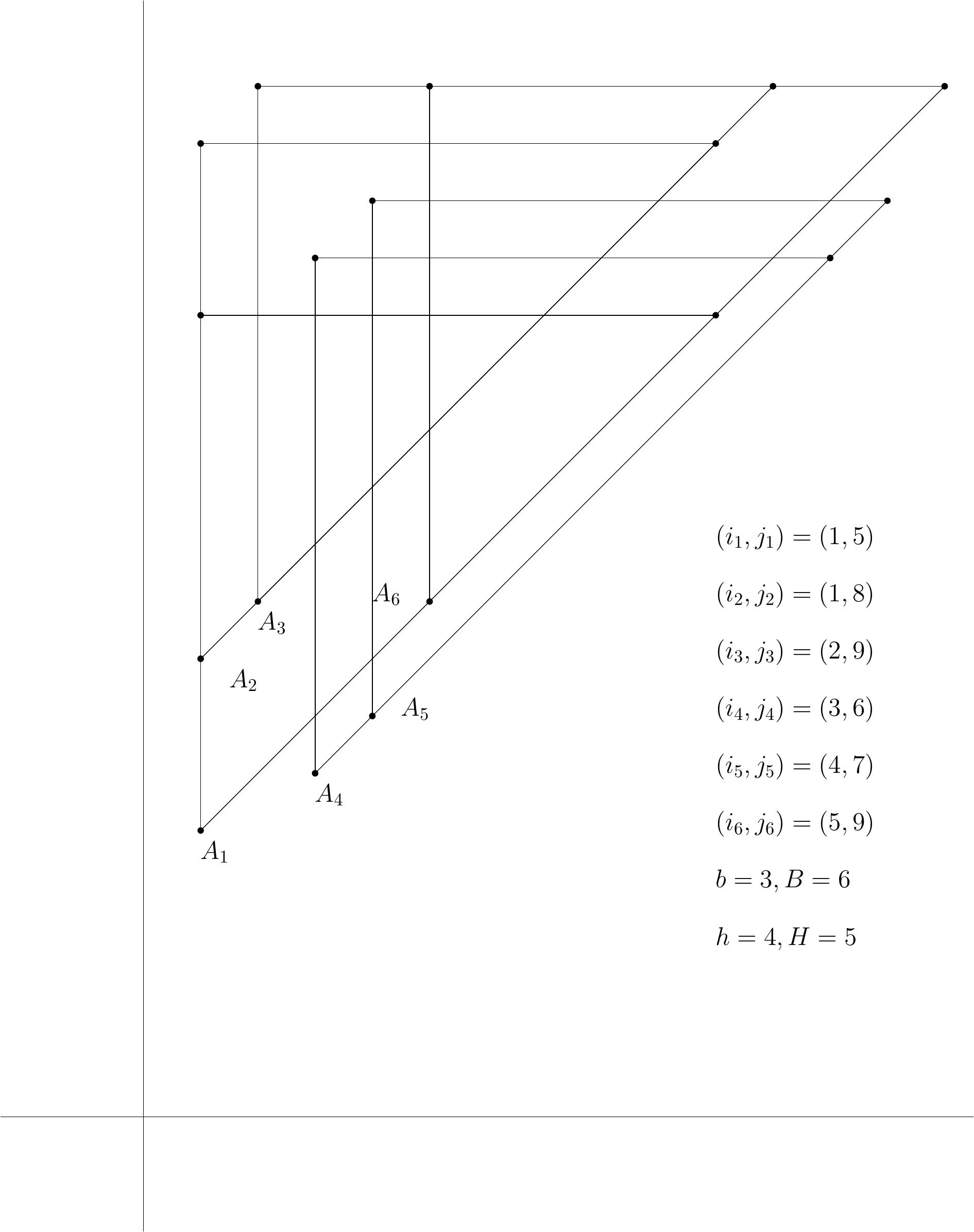}
 \caption{An example with $G_9$ having six edges}
	\label{fig_Anticycle_ex} 
\end{figure}
Running \Cref{alg_Jsets}, we get 
\begin{align*}
J_1=\{4,5\}, u_1=4, J_2=\{1\}, u_2=1, \beta =2.
\end{align*}
Running \Cref{alg_Ksets}, we get
\begin{align*}
K_1=\{4,5\}, v_1=5, K_2=\{6\}, v_2=6, \gamma=2.
\end{align*}
Note that $b=3, B=6, h=4, H=5$ and 
\[
i_b=i_3= 2 < i_h=i_4=3.
\]
Let $n=18= 2r$. Running \Cref{alg:anticycle_initialvertices} on Macaulay2 \cite{GS96} with the input $G_9$ and $n=18$, we get the sequence
\[
a_1=1, a_2=4.
\]
Running \Cref{alg:anticycle_finalvertices} with the input $G_9, n=18$ and $a_\textsf{index}=a_2=4$, noting that $i_B=i_6=5, j_B=j_6=9$, we get the sequence
\[
a_2=4, 6, 8, \ldots, 22, 24, 27.
\]
Hence $G_{27}$ has an induced anticycle $C_{13}^c$ whose consecutive vertices are
\[
1, 4, 6, 8, 10, \ldots, 22, 24, 27.
\]
Similarly, $G_{28}$ has an induced $C_{14}^c$ whose consecutive vertices are
\[
1, 4, 6, 8, 10, \ldots, 22, 24, 27, 28,
\]
while $G_{29}$ has an induced $C_{14}^c$ with consecutive vertices
\[
1, 4, 6, 8, 10, \ldots, 22, 24, 26, 29.
\]

\end{ex}

\begin{lem}
\label{lem_length_inducedanticycle}
With the notations as in \Cref{constr_inducedanticycle}, we have $m\ge 4$.
\end{lem}
\begin{proof}
Note that by construction, $a_2-a_1 =j_t-i_t-1$ for some $1\le t\le s$, hence $a_2=a_1+j_t-i_t-1 < i_b+j_t-i_t-1< 2r$. On the other hand, $a_{m-1} > n+i_B > 2r$, by the hypothesis on $n$. Hence $m\ge 4$.
\end{proof}

\begin{lem}
\label{lem_consecutivevertices}
With the notations as in \Cref{constr_inducedanticycle}, we have $\{a_p,a_{p+1}\}\notin E(G_{n+r})$ for all $p=1,\ldots,m$ with $a_{m+1}=a_1$.
\end{lem}

\begin{proof}
Since $i_B+n<a_{m-1}<a_m=j_B +n$ and $i_q\le a_1<i_b$, we get $\{a_{m-1},a_m\}$, $\{a_1,a_m\}\notin E(G_{n+r})$. Now assume that $p \in \{1,\ldots, m-2\}$. 

\noindent \textsc{Case 1: $a_p < i_h$.} In this case, $a_{p+1}$ is determined from $a_p$ by Algorithm~\ref{alg:anticycle_initialvertices}. Hence $1\le p\le d$ and $a_p\le a_d<i_h.$ Set 
\begin{align}\label{Eq_defi_t}
t=\min\{\alpha\in\{1,\ldots,\beta\} \mid i_{u_\alpha}\le a_p\}.
\end{align}
It follows that $t\ge 2 $ and $i_{u_t}\le a_p<i_{u_{a}}$ for all $1\le a\le t-1$. Moreover, we see that \[a_{p+1}-a_p=j_{u_t}-i_{u_t}-1<j_{u_t}-i_{u_t}.\] 
Now if there exists an integer $\alpha\in\{1,\ldots,s\}$  such that $(a_p,a_{p+1})\in \Delta_\alpha$ then  $i_{\alpha} \le a_p<i_{u_{t-1}}$ and $j_\alpha -i_\alpha\le a_{p+1}-a_p < j_{u_t}-i_{u_t}.$   By the definitions of $u_t$ and $J_t$, we must have 
\[
\alpha\in \cup_{i=1}^{t-1}J_i,
\]
which implies $\alpha\in J_\nu$ for some $\nu=1,\ldots,t-1$. Thus, $a_p\ge i_\alpha \ge i_{u_\nu}$.
By \eqref{Eq_defi_t}, we get $\nu \ge t$, a contradiction. This means $(a_p,a_{p+1})\notin \Delta_\alpha$ for all $\alpha=1,\ldots,s$.

\noindent \textsc{Case 2: $a_p\ge i_h$.} In this case, $a_{p+1}$ is determined from $a_p$ by Algorithm \ref{alg:anticycle_finalvertices}. Set 
 \begin{align}\label{Eq_defi_t2}
 t=\min\{\alpha\in\{1,\ldots,\gamma\} \mid a_p\le n+i_{v_\alpha}\}.
 \end{align}
 It follows that $t\ge 1$ and $a_{p+1}-a_p=j_{v_t}-i_{v_t}-1<j_{v_t}-i_{v_t}.$ If $t=1$ then
 \[
 a_{p+1}-a_p<j_{H}-i_{H}\le j_\alpha -i_\alpha
 \]
for all $\alpha=1,\ldots,s$. Hence $(a_p,a_{p+1})\notin \Delta_\alpha$ for all $\alpha=1,\ldots,s$. Now we consider the case where $t\ge 2$. By \eqref{Eq_defi_t2}, we get $n+i_{v_{t-1}}<a_p\le n+i_{v_t}$ and
\[
a_{p+1}=a_p+j_{v_t}-i_{v_t}-1\le n+i_{v_t}+ j_{v_t}-i_{v_t}-1=n+j_{v_t}-1<n+j_{v_t}.
\]
Furthermore, 
\[
a_{p+1}=a_p+j_{v_t}-i_{v_t}-1> n+i_{v_{t-1}}+j_{v_t}-i_{v_t}-1.
\]
Since $j_{v_t}-i_{v_t}>j_{v_{t-1}}-i_{v_{t-1}}$ per \Cref{constr_Jsets}, one has
\begin{align}\label{Eq_defi_t3}
a_{p+1}\ge  n+i_{v_{t-1}}+j_{v_t}-i_{v_t}>n+j_{v_{t-1}}.
\end{align}
Now if there exists an integer $\alpha\in\{1,\ldots,s\}$  such that $(a_p,a_{p+1})\in \Delta_\alpha$ then  
\[
n+j_{v_{t-1}}<a_{p+1}\le n+j_\alpha,
\]
which implies $j_\alpha>j_{v_{t-1}}.$
Combining with the fact that 
\[
j_{v_t}-i_{v_t}>a_{p+1}-a_p\ge j_\alpha-i_\alpha
\]
and the definitions of $K_t$ and $v_t$, we must have $\alpha\in \cup_{i=1}^{t-1}K_i$. This implies that $\alpha\in K_\nu$ for some $1\le \nu \le t-1$. As a consequence,
\begin{equation}
\label{eq_ineq_ap+1}
a_{p+1}\le n+j_\alpha\le n+j_{v_\nu}\le n+j_{v_{t-1}}.
\end{equation}
The third inequality in the last  chain holds because of \Cref{constr_Jsets}. For the second inequality, note that $j_\alpha-i_\alpha=j_{v_\nu}-i_{v_\nu}$, as $\alpha, v_\nu \in K_\nu$. By definition, $v_\nu= \max K_\nu\ge \alpha$, hence
\[
j_{v_\nu}-j_\alpha=i_{v_\nu}-i_\alpha \ge 0,
\]
as desired.

Now \eqref{eq_ineq_ap+1} contradicts the inequality \eqref{Eq_defi_t3}. So $(a_p,a_{p+1})\notin \Delta_\alpha$ for all $\alpha=1,\ldots,s$. 
\end{proof}

\begin{lem}
\label{lem_non_consecutivevert}
With the notations as in \Cref{constr_inducedanticycle}, we have $\{a_p,a_z\}\in E(G_{n+r})$ for any $1\le p<z\le m$ such that $z-p\ge 2$ and $(p,z)\ne (1,m)$.
\end{lem}
\begin{proof}
 We proceed through two steps.

\noindent \textsf{Step 1}: $\{a_1,a_z\}\in E(G_{n+r})$ for all $3\le z\le m-1.$

Indeed, $a_{m-1} < a_m=n+j_B=n+j_q+1$, thanks to the hypothesis, hence $a_z \le a_{m-1} \le n+j_q$.

We first consider the case where $i_b\le i_h$. Then thanks to Algorithm \ref{alg:anticycle_initialvertices}, $i_{u_{\beta}}\le a_1$. If $a_z\le j_{u_{\beta}}+n$, then $(a_1,a_z)\in \Delta_{u_{\beta}}$, as the following chain holds
\[
0\le a_1-i_{u_{\beta}}\le a_z-j_{u_{\beta}}\le n.
\]
In this chain, the middle inequality  holds because thanks to Equation \eqref{eq_a2-a1},
$$
a_z - a_1=a_z-a_2+a_2-a_1\ge 1+a_2-a_1 = j_{u_{\beta}}-i_{u_{\beta}}.
$$ 
If $a_z > n+j_{u_\beta}$,  then 
$$
j_q+n\ge a_z>j_{u_{\beta}}+n\ge 2r\ge i_b+j_q-i_q > a_1+j_q-i_q\ge j_q,
$$ 
hence $(a_1,a_z)\in \Delta_q$ because the following chains holds
\[
0\le a_1-i_{q}\le a_z-j_{q}\le n.
\]
In other words, $\{a_1,a_z\}\in E(G_{n+r})$ for all $3\le z\le m-1$ in this case. 

Next consider the case $i_h<i_b$. We show that $\{a_1,a_z\}\in E(G_{n+r})$ for any $3\le z\le m-1$  by going along the same lines by replacing ${u_{\beta}}$ by $u_1=h$. We leave the details to the interested reader. This completes the proof of Step 1.

\noindent \textsf{Step 2}: $\{a_p,a_z\}\in E(G_{n+r})$ for all $2\le p < z\le m$ and $z-p\ge 2$.

Observe that $a_p\ge i_b$. We consider two cases.

\noindent \textsc{Case 1: $a_p < i_h$.} Again let 
\[
t=\min\{\alpha\in\{1,\ldots,\beta\} \mid i_{u_\alpha}\le a_p\},
\]
then $t\ge 2$ as $u_1=h$. We have $i_{u_t}\le a_p<i_{u_{t-1}}$ and $a_{p+1}-a_p=j_{u_t}-i_{u_t}-1$ by Algorithm \ref{alg:anticycle_initialvertices}. If $a_z\le j_{u_t}+n$ then
\[
a_z-a_p =  a_z- a_{p+1} + a_{p+1}-a_p\ge j_{u_t}-i_{u_t},
\]
which gives $0\le a_p-i_{u_t}\le a_z-j_{u_t}\le n$, i.e. $(a_p,a_z)\in \Delta_{u_t}.$ On the other hand, if $a_z>j_{u_t}+n$, then we claim $(a_p,a_z)\in \Delta_{b}$. Indeed, since $a_z>2r\ge i_{u_{t-1}}+j_b-i_b>a_p+j_b-i_b$ and  $a_z\le j_b+n$, the  chain 
\[
0\le a_p-i_b \le a_z -j_b \le n
\]
is valid.

\noindent \textsc{Case 2: $a_p \ge i_h$.}  If $z=m$ then we have that $(a_p,a_m)\in \Delta_b\cup \Delta_B$, as $i_b\le a_p\le i_B+n$ for all $2\le p\le m-2$. More precisely, it can be checked that if $a_p \ge i_B$ then $(a_p,a_m) \in \Delta_B$, while if $a_p < i_B$ then $(a_p,a_m) \in \Delta_b$. If $z\le m-1$, then as
\[
a_{z-1} \le a_{m-2} \le n+i_B=n+i_{v_\gamma} \quad \text{(per \Cref{constr_inducedanticycle} and Equation \eqref{eq_vgammaB})},
\]
we may let 
\[
\mu=\min\{\alpha\in\{1,\ldots,\gamma\}\mid a_{z-1}\le i_{v_\alpha}+n\}.
\]
Then $a_{z-1}\le i_{v_\mu}+n$ and by Algorithm \ref{alg:anticycle_finalvertices}, $a_z-a_{z-1}=j_{v_\mu}-i_{v_\mu}-1$. Therefore
\[
a_z-a_p=a_z-a_{z-1}+a_{z-1}-a_p\ge j_{v_\mu}-i_{v_\mu}
\]
and
\[
a_z=a_{z-1}+j_{v_\mu}-i_{v_\mu}-1\le j_{v_\mu}+n.
\]
It follows that if $a_p\ge i_{v_\mu}$, then $(a_p,a_z)\in\Delta_{v_\mu}$, since the following chain holds
\[
0\le a_p-i_{v_\mu}\le a_z-j_{v_\mu}\le n.
\]
Finally, it remains to consider the case $i_b\le a_p<  i_{v_\mu}$.  If $a_z\ge j_b -i_b +i_{v_\mu}$, then
\[
a_z-a_p >a_z-i_{v_\mu}\ge j_b -i_b,
\]
which yields 
\[
0\le a_p-i_{b}\le a_z-j_{b}\le n.
\]
In other words, $(a_p,a_z)\in \Delta_b$.

Assume that $a_z< j_b -i_b +i_{v_\mu}$. In this case, the definition of $\mu$ implies $\mu=1$, hence
\[
a_z-a_p=a_z-a_{z-1}+a_{z-1}-a_p\ge j_H-i_H=j_h-i_h.
\]
Since $n\ge 2r\ge j_b -i_b +i_{v_\mu}-j_h>a_z-j_h,$  the following chain holds
\[
0\le a_p-i_{h}\le a_z-j_{h}\le n, \; \text{i.e,}\; (a_p,a_z)\in \Delta_{h}.
\]
This also finishes the proof of the lemma.
\end{proof}

\begin{proof}[Proof of \Cref{prop_reg3_specialcase1}]
By Fr\"oberg's theorem, we have to show that $G_{n+r}$ is not cochordal for all $n\ge 2r$. It suffices to construct an induced anticycle $C_m^c$ in $G_{n+r}$, for some $m\ge 4$. Keep using notations as in \Cref{constr_inducedanticycle}. We wish to show that the vertices $a_1,\ldots, a_{m}$ in \Cref{constr_inducedanticycle} form an induced anticycle in the graph $G_{n+r}$. 

This follows from combining Lemmas \ref{lem_length_inducedanticycle}, \ref{lem_consecutivevertices}, and \ref{lem_non_consecutivevert}. Therefore $\reg I_n\ge 3$ for all $n\ge 3r$.
\end{proof}

\section{A sharp regularity bound}
\label{sect_sharpbound}

The next theorem gives our first major result on the regularity of invariant chains of edge ideals.
\begin{thm}
\label{thm_reg<=3}
 Let $\Icc=(I_n)_{n\ge 1}$ be an $\Inc$-invariant chain of eventually nonzero edge ideals. Let $r=\ind(\Icc)\ge 1$. Then for all $n\ge 4r$, there is an inequality $\reg I_n\le 3$.
\end{thm}
The key to this theorem is the following
\begin{lem}
\label{lem_cochordal_filtration}
Keep using notations of \Cref{thm_reg<=3}. Let $G_n$ be the graph corresponding to $I_n$. Let $H$ be an induced subgraph of $G_n$ with at least one edge. Set
\[
u_1=\max\{u\in V(H)\mid \text{for some $u<v, v\in V(H)$, we have} ~ \{u,v\}\in E(H)\}.
\] 
Let $v_1\in V(H)$ be any vertex such that $u_1 <v_1$ and $\{u_1,v_1\} \in E(H).$ Then for every $n\geq 4r$ and every choice of $H$ and $v_1$ as above, the graph $H\setminus N_H[v_1]$ is cochordal.
\end{lem}
\begin{proof}
Assume the contrary that for some $n\ge 3r$ and some induced subgraph $H$ of $G_{n+r}$ with $E(H)\neq \emptyset$, the graph  $H\setminus N_H[v_1]$ is not cochordal.

Assume that $E(G_r)=\{\{i_1,j_1\},\ldots,\{i_s,j_s\}\}$ where $s\ge 1$, $1\le i_p < j_p\le r$ for all $p=1,\ldots,s$. For each $1\le t\le s$, denote $\Delta_t=\Delta((i_t,j_t),n)$ for simplicity. Therefore, 
$$E(H)\subseteq E(G_{n+r})=\{\{u,v\}: (u,v)^\le \in \cup_{t=1}^s\Delta_t\}.$$
Hence there exists an integer $\ell\in\{1,\ldots,s\}$ such that $(u_1,v_1)\in\Delta_\ell$. Observe that $(i_\ell+j, v_1) \in \Delta_\ell$ for all $0\le j \le u_1-i_\ell$. Hence
\[
\{v_1, i_\ell, i_\ell+1,\ldots,u_1\} \subseteq N_{G_{n+r}}[v_1].
\]
Since $H$ is an induced subgraph of $G_{n+r}$, it holds that 
$$
N_H[v_1]=V(H) \cap N_{G_{n+r}}[v_1].
$$
Therefore $V(H\setminus N_H[v_1])\subseteq \{1,2,\ldots,i_\ell-1\}\cup V_1$, where $i_\ell+j\notin V_1$ for all $0\le j \le u_1-i_\ell$, in particular, if $u\in V_1$ then $u> u_1$. If follows from the maximality of $u_1$ in the hypothesis that:
\begin{equation}
\label{eq_ineq_smallu}
\text{If $u < v$ and $\{u,v\} \in E(H\setminus N_H[v_1])$ then $u\leq i_\ell-1$}. 
\end{equation}

\textsf{Step 1:} We claim that $H\setminus N_H[v_1]$ has no induced $2K_2$. 

Note that $H\setminus N_H[v_1]$ is also an induced subgraph of $G_{n+r}$. By Lemma~\ref{lem_twopoints_induced2K2}, the graph $H\setminus N_H[v_1]$ has no induced $2K_2$. Indeed, assume that $\{u_2,v_2\}, \{u_3,v_3\}$ is an induced $2K_2$ of $H\setminus N_H[v_1]$ where $u_p<v_p$. Then \Cref{lem_twopoints_induced2K2} implies that for some non-negative number $j_{2s}$, $u_3 > n+j_{2s} >n$. This contradicts \eqref{eq_ineq_smallu} and the fact that $n\ge 3r >i_\ell$.

\textsf{Step 2:} Since $H\setminus N_H[v_1]$ is not cochordal, and yet does not contained any induced $2K_2= C_4^c$, it necessarily contains an induced anticycle $C_m^c$, where $m\ge 5$. Label the consecutive vertices of $C_m^c$ as $a_1,\ldots,a_m$ (following the terminology of \Cref{subsect_edgeideals}). By reindexing, we may assume that $a_1=\min\{a_p: 1\le p\le m\}$. Let $a_{m+1}=a_1$.

For each $1\le p <q\le m$ such that $q-p\ge 2, (p,q)\neq (1,m)$, it holds that $\{a_p,a_q\} \in E(H\setminus N_H[v_1])$, so $(a_p,a_q)^\le\in \Delta_{pq}$, where $\Delta_{pq}=\Delta((i_{pq},j_{pq}),n)$ is one of the sets $\Delta_1,\ldots,\Delta_s$. As in the proof of \Cref{lem_inducedanticyc}, we have

\noindent\textbf{Claim 1:}  For all $1\le t\le m$, none of the points $(a_t,a_{t+1}), (a_{t+1},a_t)$ belongs to 
$$
\Gamma= \mathop{\bigcup_{1\le p< q\le s}}_{q-p\ge 2, (p,q)\neq (1,m)} \Delta_{pq}.
$$

Consider two cases according to whether $a_2\le a_m$ or $a_m< a_2$. 

\noindent \textsf{\bf Case I:} $a_2\le a_m$. Thanks to \Cref{lem_inducedanticyc}, we get $a_m=\max\{a_p: 1\le p\le m\}$. 

Apply \Cref{lem_twopoints_corners} for the points $(a_1,a_3)\in \Delta_{13}$ and $(a_2,a_m) \in \Delta_{2m}$. The conclusion is that $a_1< i_{2m}$ and
\begin{equation}
\label{eq_ineq_a{m}large}
a_3 \le j_{13}+n < a_m.
\end{equation}

\noindent\textsf{Case I.1:} $a_2 \le a_{m-1}$. By Claim 1, $(a_{m-1},a_m)\notin \Delta_{2m}$, hence the following chain is false
\[
0\le a_{m-1}-i_{2m} \le a_m-j_{2m} \le n.
\]
The first inequality in the chain holds since $a_{m-1}\ge a_2\ge i_{2m}$, and the last inequality also holds since $a_m\le j_{2m}+n$. Hence the middle inequality is false, namely
\begin{equation}
\label{eq_ineq_a{m_1}large}
a_{m-1}>a_m+i_{2m}-j_{2m}.
\end{equation}
By Claim 1, $(a_{m-2},a_{m-1})\notin \Delta_{1(m-1)}$, hence the following chain is false
\[
0\le a_{m-2}-i_{1(m-1)} \le a_{m-1}-j_{1(m-1)} \le n.
\]
The first inequality in the chain holds since $a_{m-2}\ge a_1\ge i_{1(m-1)}$, and the last inequality also holds since $a_{m-1}\le j_{1(m-1)}+n$. Hence the middle inequality is false, namely
\begin{equation}
\label{eq_ineq_a{m_2}large}
a_{m-2}>a_{m-1}+i_{1(m-1)}-j_{1(m-1)}.
\end{equation}
Combining \eqref{eq_ineq_a{m}large}, \eqref{eq_ineq_a{m_1}large}, and \eqref{eq_ineq_a{m_2}large},  we obtain
\begin{align*}
a_{m-2}>n+j_{13}+i_{2m}-j_{2m}+i_{1(m-1)}-j_{1(m-1)}.
\end{align*}
Since $a_{m-2} \le a_m$, $a_{m-2}\neq a_m$, and $\{a_{m-2},a_m\} \in E(H\setminus N_H[v_1])$, we obtain $a_{m-2}\le i_\ell-1$ by \eqref{eq_ineq_smallu}. Thus
\[
i_\ell-1 \ge a_{m-2}>n+j_{13}+i_{2m}-j_{2m}+i_{1(m-1)}-j_{1(m-1)},
\]
so that 
$$
n< i_\ell-1 -j_{13}-i_{2m}+j_{2m}-i_{1(m-1)}+j_{1(m-1)}\le 3r,
$$
where the last inequality follows from $i_\ell < j_\ell \le \max\{j_1,\ldots,j_s\}\le r$. This contradicts the hypothesis $n\ge 3r.$

\noindent\textsf{Case I.2:} $a_2 > a_{m-1}$. Looking at the points $(a_1,a_{m-2})=(a_1,a_{m-2})^\le \in \Delta_{1(m-2)}$ and $(a_{m-1},a_2)=(a_2,a_{m-1})^\le \in \Delta_{2(m-1)}$. Note that $a_1 \le a_{m-1}$, and $(a_1,a_2)$, $(a_{m-1},a_{m-2})\notin \Delta_{1(m-2)} \cup \Delta_{2(m-1)}$ thanks to Claim 1, so arguing as above using \Cref{lem_twopoints_corners}, we deduce $i_{1(m-2)} < i_{2(m-1)}.$ Moreover, the same result yields 
\begin{equation}
\label{eq_ineq_a_{2}large}
a_2>n+j_{1(m-2)}.
\end{equation}
Since $\{a_2,a_m\} \in E(H\setminus N_H[v_1])$, $a_2\le a_m$ and $a_2\neq a_m$, we get  from \eqref{eq_ineq_smallu} that $a_2 \le i_\ell-1$. Combining with \eqref{eq_ineq_a_{2}large},  we obtain
\[
i_\ell -1 \ge a_2>n+j_{1(m-2)},
\]
hence 
$$
n<i_\ell-1-j_{1(m-2)} < r.
$$
This contradicts the hypothesis $n\ge 3r.$

\noindent \textsf{\bf Case II:} $a_2> a_m$. Relabel the vertices as follows: $b_1=a_1, b_i=a_{m+2-i}$ for $2\le i\le m$. Then $b_1=\min\{b_1,\ldots,b_m\}$ and $b_2 < b_m$. Arguing as for Case I, we finish Case II and conclude that $H\setminus N_H[v_1]$ is cochordal. The proof is completed.
\end{proof}

\begin{proof}[Proof of \Cref{thm_reg<=3}]
For a graph $H$, by abuse of notation, write $\reg H$ for the regularity of the ideal $I(H)$. Let $n\ge 3r$ be an integer. We prove that if $H$ is an induced subgraph of $G_{n+r}$, then $\reg H\le 3$.  Induction on the cardinality of $E(H)$.

If $E(H)=\emptyset$ then $\reg H= \reg (0) =-\infty$. Assume that $E(H)\neq \emptyset$. Let $u_1$ be the maximal value of $u\in V(H)$ such that for some $u<v, v\in V(H)$, $\{u,v\}$ belongs to $E(H)$. Let $v_1\in V(H)$ be a neighbor of $u_1$ in $H$ such that $v_1>u_1$. By \Cref{lem_cochordal_filtration} and Fr\"oberg's theorem, $\reg(H\setminus N_H[v_1]) \le 2$. This together with \Cref{lem_reg_edgeideal_deletion}, yields the chain
\[
\reg(H) \le \max\{\reg (H\setminus N_H[v_1])+1, \reg (H\setminus v_1)\} \le \max\{3, \reg (H\setminus v_1)\}.
\]
Clearly $\{u_1,v_1\} \in E(H) \setminus E(H\setminus v_1)$. Also $H\setminus v_1$ is an induced subgraph of $G_{n+r}$, hence by the induction hypothesis $\reg (H\setminus v_1)\le 3$. Therefore $\reg(H) \le 3$. This finishes the induction and choosing $H=G_{n+r}$, we get the desired conclusion.
\end{proof}

\section{Convergence}
\label{sect_converge}

Now we state our main accomplishment in this paper. In the following statement, we employ the $q$-invariant introduced in \Cref{defn_q-invariant}. 
\begin{thm}
\label{thm_convergence}
Let $\Icc=(I_n)_{n\ge 1}$ be an $\Inc$-invariant chain of eventually nonzero edge ideals. Keep using \Cref{notn_chainofedgeids}. Then for all $n\ge N=\max\{5r, 2r(r-2), 4(r+q(I_r))\}$, there is an equality
\[
\reg I_n = \reg I_{n+1},
\]
and moreover $\lim\limits_{n\to \infty} \reg I_n \in \{2,3\}$.

In addition, the following statements are equivalent
\begin{enumerate}[\quad \rm (1)]
 \item $\reg I_n=2$ for all $n\gg 0$;
 \item $\reg I_N=2$;
 \item  Either $j_q = \max\{j_1,\ldots,j_s\}$, or it holds that $\min \{j_t-i_t: t=1,\ldots,s\}=1$ and $\indmatch(G_{3r})=1$.
\end{enumerate}
\end{thm}
\begin{rem}
\label{rem_Nr}
We note that the number $N=\max\{ 5r, 2r(r-2), 4(r+q(I_r))\}$ in \Cref{thm_convergence} can be replaced by a constant depending only on $r$. Indeed, let $p=\Msupp I_r \le r$, since $I_r$ is quadratic and nonzero, by \Cref{defn_q-invariant},
\[
q(I_r)= \sum_{i=0}^2 \dim_\kk \left(\frac{R_p}{I_r\cap R_p}\right)_i \le 1+p+\binom{p+1}{2}-1 \le \dfrac{r^2+3r}{2}.
\]
Therefore $4(r+q(I_r))\le 2(r^2+5r)$, and thus $N\le 2(r^2+5r)$. Hence we can replace $N$ by $2(r^2+5r)$ in the statement of \Cref{thm_convergence}. While we do not know the optimal (minimal) value for the number $N$ such that $\reg I_n=\reg I_{n+1}$ for all $n\ge N$,  \Cref{rem_Nr_lowerbound} below shows that this value has to be at least $3r-9$.
\end{rem}

For the proof of \Cref{thm_convergence}, we start in Propositions \ref{prop_reg2_sufficient_consecutivevertices} and \ref{prop_reg2_all-embracinginterval} with the sufficient conditions for the equality $\reg I_n=2$ to hold for all $n\gg 0$.
\begin{prop}
\label{prop_reg2_sufficient_consecutivevertices}
 Keep using \Cref{notn_chainofedgeids}. If $\indmatch(G_{3r})=1$  and $\min \{j_p-i_p: p=1,\ldots,s\}=1$, then $\reg I_n=2$ for all $n\ge \max\{5r, 2r(r-2)\}$.
\end{prop}
\begin{proof}
By Fr\"oberg's theorem, we have to show that $G_n$ is cochordal for all $n \ge \max\{5r, 2r(r-2)\}$. 

By the hypothesis, $I_n$ is a nonzero edge ideal for all $n\ge r$. Thanks to \Cref{thm_no3K2}, $\indmatch(G_n)=1$ for all $n\ge 3r$. And hence $G_n$ does not contain any induced $2K_2$ for all $n\ge 3r$. Take $n\ge \max\{5r, 2r(r-2)\}$. We show that $G_n$ is cochordal.

It suffices to prove that $G_n$ does not contain any induced $C_m^c$ for $m\ge 5$. Assume the contrary, that for some $m\ge 5$ and some $n\ge \max\{5r, 2r(r-2)\}$, $G_n$ contain an induced $C_m^c$. By \Cref{prop_noCmc}, 
\begin{equation}
 \label{eq_ineq_mlarge}
 m> \dfrac{n}{r} \ge \max\{5, 2(r-2)\}.
\end{equation}
By the hypothesis, for some $1\le p\le s$, we have $j_p=i_p+1$. Since
$$
\Delta((i_p,i_p+1),n-r)=\{(u,v): i_p\le u<v \le i_p+1+n-r\},
$$
we deduce that $A=\{i_p,i_p+1,\ldots,i_p+n-r+1\}$ is a clique in $G_n$. Note that $|V(G_n\setminus A)| \le n-(n-r+2)=r-2$.

Denote the consecutive vertices of the induced $C_m^c$ as a subgraph of $G_n$ by $a_1,a_2,\ldots,a_m$. Denote $a_{m+1}=a_1$. So $a_i, a_{i+1}$ are not adjacent for all $1\le i\le m$. Since $A$ is a clique and $C_m^c$ is an induced subgraph of $G_n$, this implies $\{a_1,\ldots,a_m\} \cap A$ has at most $m/2$ elements. Hence
\[
m = |\{a_1,\ldots,a_m\} \cap A| + |\{a_1,\ldots,a_m\} \setminus A| \le \dfrac{m}{2} + |V(G_n\setminus A)| \le \dfrac{m}{2}+(r-2).
\]
In particular, $m\le 2(r-2)$, contradicting \eqref{eq_ineq_mlarge}. Therefore the above assumption is wrong, and $G_n$ is cochordal for all $n\ge \max\{5r, 2r(r-2)\}$. This concludes the proof.
\end{proof}

\begin{prop}
\label{prop_reg2_all-embracinginterval}
Keep using \Cref{notn_chainofedgeids}. If $j_q=\max\{j_1,\ldots,j_s\}$ then $\reg I_n=2$ for all $n\ge 3r$.
\end{prop}

\begin{proof}
By Fr\"oberg's theorem, we have to show that $G_{n+r}$ is cochordal for all $n\ge 2r$. Assume the contrary, that for some $m\ge 4$ and some $n\ge 2r$, $G_{n+r}$ contain an induced $C_m^c$.
	
Label the consecutive vertices of the induced $C_m^c$ by $a_1,a_2,\ldots,a_m$ in such a way that $a_1=\min\{a_1,\ldots,a_m\}.$

We may assume that $a_2\le a_m$ (when $a_2< a_m$, relabel the vertices as $b_1=a_1, b_i=a_{m+2-i}$ for $2\le i\le m$). 

By the hypothesis $j_q=\max\{j_1,\ldots,j_s\}$ and $i_1=\cdots=i_q=\min\{i_1,\ldots,i_s\}$, one has $i_q\le a_i\le j_q+n$ for  $i=1,\ldots,m$. Since $(a_1,a_m)\notin \Delta_q=\Delta((i_q,j_q),n)$,  following chain is false
\begin{align*}
0\le a_1-i_q\le a_m-j_q\le n.
\end{align*}
It follows that $a_1> a_m+i_q-j_q.$ Assume that $(a_1,a_3)\in \Delta_b, (a_2,a_m)=(a_2,a_m)^\le \in \Delta_c$ and $(a_1,a_{m-1})\in \Delta_d$, with $b,c,d \in \{1,2,\ldots,s\}$ (note that $b=d$ if $m=4$).

 Consider the points $(a_1,a_3)\in \Delta_b$ and $(a_2,a_m)  \in \Delta_c$. Since $\{a_1,a_m\}, \{a_2,a_3\}\notin E(G_{n+r})$, we have $(a_1,a_m), (a_2,a_3)\notin \Delta_b \cup \Delta_c$. Arguing as in \Cref{lem_cochordal_filtration}, using Lemma~\ref{lem_twopoints_corners} for the two points $(a_1,a_3), (a_2,a_m)$ and the fact that $a_1 <a_2$, we get $i_b<i_c$, moreover $a_1<i_c\le r$ and $a_m>n+j_b$. Thus
\[
r>a_1>a_m+i_q-j_q>n+j_b+i_q-j_q,
\]
which implies $n<r+j_q-j_b-i_q<2r$. This contradicts the hypothesis $n\ge 2r.$ Therefore the above assumption is wrong, and $G_{n+r}$ is cochordal for all $n\ge 2r$. This concludes the proof.
\end{proof}
\begin{rem}
\label{rem_Nr_lowerbound}
The lower bound $3r$ in the statement of \Cref{prop_reg2_all-embracinginterval} is nearly sharp. Let $r\ge 5$ be an integer, $(I_n)_{n\ge 1}$ be the invariant chain with stability index $r$ and $I_r=(x_1x_r,x_{r-3}x_{r-1})$. The hypothesis of \Cref{prop_reg2_all-embracinginterval} is clearly satisfied. We claim that $\reg I_{3r-10} \ge 3$. Indeed, it is not hard to check that $\{2r-8,2r-6\}, \{r-4,3r-10\}$ form a $2K_2$ in $G_{3r-10}$, so \Cref{lem_reg_indmatch} implies the desired inequality. Thus the lower bound $3r$ in the statement ``$\reg I_n=2$ for all $n\ge 3r$'' cannot be improved to $3r-10$. 
\end{rem}

Using the results in \Cref{sect_longanticycles}, we complement Propositions \ref{prop_reg2_sufficient_consecutivevertices} and \ref{prop_reg2_all-embracinginterval} by giving a condition guaranteeing that $\reg I_n=3$ for $n\gg 0$. 
\begin{thm}
\label{thm_reg3}
Keep using \Cref{notn_chainofedgeids}. Assume $\max\{j_1,\ldots,j_s\}\ge j_q+1$ and $\min\{j_p-i_p: p=1,\ldots,s\}\ge 2$.  Then for all  $n\ge 4(r+q(I_r))$, there is an equality $\reg(I_n)=3$.
\end{thm}

\begin{proof}
The hypothesis implies that $r\ge 4$. We proceed by induction on $q(I_r)\ge 0$. The case $q(I_r)=0$ is vacuous: in that case $I_r=(1)$ and the hypothesis is not satisfied.

Assume that $q(I_r)>0$. Let $p=\Msupp(I_r)=\max\{j_1,\ldots,j_s\}$. If $p=j_q+1$, then by \Cref{prop_reg3_specialcase1}, we deduce that $\reg I_n\ge 3$ for all $n\ge 3r$. Together with \Cref{thm_reg<=3}, we get $\reg I_n= 3$ for all $n\ge 4(r+q(I_r))$.

Now assume that $p\ge j_q+2$.

Let $\Jcc=(J_n)_{n\ge 1}$ be the chain defined in \Cref{prop_intersection_chain}. By the last result, $\ind(\Jcc)=r+1$ and since $\langle I_r \rangle_{R_{r+1}}\subseteq J_{r+1}$ are quadratic monomial ideals with 
$$
\Msupp(J_{r+1})=\Msupp(\langle I_r \rangle_{R_{r+1}})=p,
$$
\Cref{lem_q-invariant} implies that $q(J_{r+1})\le q(\langle I_r \rangle_{R_{r+1}})=q(I_r)$.

We claim that $q(J_{r+1})< q(I_r)$. Assume the contrary, then \Cref{lem_q-invariant} implies that $J_{r+1}=\langle I_r \rangle_{R_{r+1}}$. By \Cref{prop_intersection_chain}, we get $\Icc$ is a quasi-saturated chain. But now $x_{i_1}x_{j_q} \in I_r$, so $x_{i_1}x_{j_q+1} \in I_{r+1}$ by $\Inc$-invariance. Since $\Msupp(I_r)=\max\{j_1,\ldots,j_s\} \ge j_q+1$ and $\Icc$ is quasi-saturated, we deduce
\[
x_{i_1}x_{j_q+1} \in I_{r+1}\cap \kk[x_1,\ldots,x_p]=I_r \cap \kk[x_1,\ldots,x_p].
\]
Thus $x_{i_1}x_{j_q+1} \in I_r$, namely $\{i_1,j_q+1\} \in E(G_r)$. This contradicts the fact that $j_q=\max\{j_t \mid i_t=i_1\}$. Therefore $q(J_{r+1})< q(I_r)$.

Let $H_n$ be the graph corresponding to $J_n$ for each $n$.

\textbf{Observation:} There is an equality
\begin{gather*}
E(H_{r+1})=\left\{\{i_t,j_t\},\{i_t,j_t+1\},\{i_t+1,j_t+1\} \mid 1\le t\le s, j_t <p \right \} \quad \\
 \bigcup \quad \left\{\{i_t, j_t\} \mid 1\le t\le s, j_t=p\right\}.
\end{gather*}
This follows from inspecting the definition
\[
J_{r+1}=\langle I_{r+1} \cap R_p \rangle_{R_{r+1}}
\]
and the fact that $I_{r+1}=\langle \Inc_{r,r+1}(I_r) \rangle_{R_{r+1}}$.

From the Observation and the assumption that $p\ge j_q+2$, we see that the chain $\Jcc$ also satisfies the hypothesis of \Cref{thm_reg3}. In details,
\begin{align*}
 &\min\{v-u \mid u<v, \{u,v\} \in E(H_{r+1})\}  =  \min\{v-u \mid u<v, \{u,v\} \in E(G_r)\} \ge 2,\\
 &\max\{v \mid \{i_1,v\} \in E(H_{r+1})\} = j_q+1,\\
 &\max\{v \mid \text{for some $u<v$ we have $\{u,v\}\in E(H_{r+1})$} \} =p.
\end{align*}
By the induction hypothesis for the chain $\Jcc$, we get for all $n\ge 4(r+1+q(J_{r+1}))$ that
\[
\reg J_n=3.
\]
Since $q(J_{r+1})<q(I_r)$, we have $4(r+q(I_r)) \ge 4(r+1+q(J_{r+1}))$. For each $n\ge 4(r+q(I_r)) \ge \max \{4(r+1+q(J_{r+1})),r+1\}$, we get
\[
\reg I_n \ge \reg(I_n+(x_{n-r+p})) =\reg J_n =3.
\]
The inequality in the chain follows from \Cref{lem_reg_mon}. The first equality in the last chain is due to \Cref{prop_intersection_chain}. Together with \Cref{thm_reg<=3}, we get
\[
\reg I_n=3 \quad \text{for each $n\ge 4(r+q(I_r))$}.
\]
This finishes the induction and the proof.
\end{proof}
We are ready for the 
\begin{proof}[Proof of \Cref{thm_convergence}]
First assume that $\indmatch(G_{3r})=2$. Then by \Cref{thm_no3K2}, we deduce $\indmatch(G_n)=2$ for all $n\ge 3r$. Hence using \Cref{lem_reg_indmatch}, for all such $n$,
\[
\reg I_n \ge 1+\indmatch(G_n)=3.
\]
Together with \Cref{thm_reg<=3}, we get $\reg I_n=3$ for all $n\ge 4r$.

Next assume that $\indmatch(G_{3r})=1$. Then by \Cref{thm_no3K2}, $\indmatch(G_n)=1$ for all $n\ge 3r$.

If  $\min \{j_t-i_t: t=1,\ldots,s\}=1$ then by \Cref{prop_reg2_sufficient_consecutivevertices}, $\reg I_n=2$ for all $n\ge \max\{5r, 2r(r-2)\}$.

If  $j_q = \max\{j_1,\ldots,j_s\}$ then $\reg I_n=2$ for all $n\ge 3r$ by \Cref{prop_reg2_all-embracinginterval}.

It remains to consider the case  $\max\{j_1,\ldots,j_s\}\ge j_q+1$ and $\min \{j_t-i_t: t=1,\ldots,s\}\ge 2$. In this case, by \Cref{thm_reg3}, we deduce that $\reg I_n=3$ for all $n\ge 4(r+q(I_r))$. All the desired assertions follow from the above arguments.
\end{proof}

\section*{Acknowledgments}
This work is partially supported by the Vietnam Academy of Science and Technology (grants CSCL01.01/22-23 and NCXS02.01/22-23). The third author is  supported by the Vietnam Ministry of Education and Training under grant number B2022-DHH-01. Parts of this work were carried out during a stay
of the authors at the Vietnam Institute for Advanced Study in Mathematics (VIASM). They would like to thank VIASM for its hospitality and generous support.


\begin{thebibliography}{99}
\bibitem{AH07}
M. Aschenbrenner and C.J. Hillar,
\emph{Finite generation of symmetric ideals}.
Trans. Amer. Math. Soc. {\bf 359} (2007), no. 11, 5171--5192.

\bibitem{CHH+19}
G. Caviglia, H.T. H\`a, J. Herzog, M. Kummini, N. Terai and N.V. Trung, 
\emph{Depth and regularity modulo a principal ideal}. 
J. Algebr. Comb. {\bf 49} (2019), 1--20.

\bibitem{CEF15}
T. Church, J.S. Ellenberg, and B. Farb, 
\emph{FI-modules and stability for representations of symmetric groups}. 
Duke Math. J. {\bf 164} (2015), no. 9, 1833--1910.



\bibitem{Co87}
D.E. Cohen,
\emph{Closure relations, Buchberger's algorithm, and polynomials in infinitely many variables}.
In: \emph{Computation theory and logic}, 78--87, Lecture Notes in Comput. Sci., {\bf 270}, Springer, Berlin, 1987.

\bibitem{DHS13}
H.L. Dao, C. Huneke, and J. Schweig, 
\emph{Bounds on the regularity and projective dimension of ideals associated to graphs.}
J. Algebraic Combin. {\bf 38}, no. 1 (2013), 37--55.

\bibitem{DEF}
J. Draisma,  R.H. Eggermont, and A. Farooq,
\emph{Components of symmetric wide-matrix varieties}.
J. Reine Angew. Math., vol. {\bf 2022}, no. 793, pp. 143--184. 


\bibitem{Fr90}
R. Fr\"oberg, 
\emph{On Stanley-Reisner rings}. In: \emph{Topics in Algebra}, vol. {\bf 26}, Part 2, pp. 57--70. Banach Center Publications (1990).

\bibitem{GS96}
D. Grayson and M. Stillman,
\emph{Macaulay2, a software system for research in algebraic geometry}.
Available at \newblock \verb|http://www.math.uiuc.edu/Macaulay2/|.

\bibitem{GN18} 
S. G\"unt\"urk\"un and U. Nagel, 
\emph{Equivariant Hilbert series of monomial orbits}, 
Proc.\ Amer.\ Math.\ Soc.\ {\bf 146}, no. 6 (2018), 2381--2393. 

\bibitem{HS12}
 C.J. Hillar and S. Sullivant, 
 \emph{Finite Gr\"obner bases in infinite dimensional polynomial rings and applications}.
Adv. Math. {\bf 229} (2012), no. 1, 1--25.

\bibitem{KLR22}
T. Kahle, D.V. Le, and T. R\"omer, 
\emph{Invariant chains in algebra and discrete geometry}. 
SIAM J. Discrete Math. {\bf 36} (2022), no. 2, 975--999.


\bibitem{Ka06}
M. Katzman,
\emph{Characteristic-independence of Betti numbers of graph ideals}. 
J. Combin. Theory Ser. A {\bf 113}, 435--454 (2006).


\bibitem{LNNR2}
D.V. Le, U. Nagel, H.D. Nguyen, and T. R\"omer,
\emph{Castelnuovo--Mumford regularity up to symmetry}.
Int. Math. Res. Not. Vol. {\bf 2021}, No. {\bf 14}, 11010--11049.

\bibitem{LN22}
D.V. Le and H.D. Nguyen,
\emph{On regularity and projective dimension up to symmetry}.
Preprint (2022), arXiv:2206.15141 [math.AC].

\bibitem{MV12}
 S. Morey and R. Villarreal, 
 \emph{Edge ideals: algebraic and combinatorial properties}. in \emph{Progress in Commutative Algebra, Combinatorics and Homology, Vol. 1} 
 (C. Francisco, L. C. Klingler, S. Sather-Wagstaff and J. C. Vassilev, Eds.), De Gruyter, Berlin (2012), pp. 85--126.


\bibitem{Mu20}
S. Murai,
\emph{Betti tables of monomial ideals fixed by permutations of the variables}.
Trans. Amer. Math. Soc. {\bf 373} (2020), 7087--7107.

\bibitem{MR22}
S. Murai and C. Raicu,
\emph{An equivariant Hochster's formula for $\mathfrak{S}_n$-invariant monomial ideals}.
J. London Math. Soc. {\bf 105} (2022), 1974--2010.


\bibitem{NR17}
U. Nagel and T. R\"omer,
\emph{Equivariant Hilbert series in non-noetherian polynomial rings}.
J. Algebra {\bf 486} (2017), 204--245.

\bibitem{Ra21}
C. Raicu, 
\emph{Regularity of $\mathfrak{S}_n$-invariant monomial ideals}.
J. Combin. Theory Ser. A {\bf 177}, 2021.

\bibitem{Ram20}
E. Ramos,
\emph{An application of the theory of FI-algebras to graph configuration spaces}.
Math. Z. {\bf 294}, 1--15 (2020).

\bibitem{RW19}
E. Ramos and G. White,
\emph{Families of nested graphs with compatible symmetric-group actions}.
Sel. Math. New Ser. {\bf 25}, 70 (2019).

\bibitem{SS17}
S. Sam and A. Snowden, 
\emph{Gr\"obner methods for representations of combinatorial categories}. 
J. Amer. Math. Soc. {\bf 30} (2017), no. 1, 159--203.

\bibitem{Wo14}
R. Woodroofe,
\emph{Matchings, coverings, and Castelnuovo–Mumford regularity}. 
J. Commut. Algebra {\bf 6}(2), 287--304 (2014).
\end{thebibliography}
\end{document}